\documentclass[a4paper,12pt]{article}
\usepackage{amsmath, amsthm, amsfonts, amssymb, bbm}
\usepackage{graphicx, psfrag}
\usepackage[usenames,dvipsnames]{xcolor}
\usepackage{graphicx}
\usepackage{ulem}
\usepackage{cite}
\usepackage{hyperref}
\usepackage{stackrel}
\usepackage{framed,color}
\definecolor{shadecolor}{rgb}{0.93, 0.79, 0.69}
\usepackage[utopia]{mathdesign}
\usepackage{mathtools}
\usepackage{mhchem}
\usepackage{gensymb}

\definecolor{columbiablue}{rgb}{0.61, 0.87, 1.0}
\definecolor{mycolor}{rgb}{.1, .7, 1}
\usepackage[margin=0.9in]{geometry}

\usepackage{amsmath,lineno,centernot}
\usepackage{soul}
\usepackage{enumerate}
\usepackage{titlesec}
\usepackage{cancel}

\usepackage{ulem}

\usepackage{mathtools}

\newcommand{\eps}{\epsilon}
\newtheorem*{acknowledgements}{Acknowledgements}
\newcommand{\und}{\underline}

\usepackage{tikz}
\usetikzlibrary{arrows,decorations.pathmorphing,backgrounds,positioning,fit,petri}
\usetikzlibrary{shapes}
\usetikzlibrary{arrows.meta}

\DeclareMathAlphabet{\mathpzc}{OT1}{pzc}{m}{it}

\newtheorem{proposition}{Proposition}[section]
\newtheorem{theorem}[proposition]{Theorem}
\newtheorem{lemma}[proposition]{Lemma}
\newtheorem{definition}[proposition]{Definition}
\newtheorem{corollary}[proposition]{Corollary}

\newtheorem{cla}[proposition]{Claim}

\newtheorem{remark}[proposition]{Remark}

\numberwithin{equation}{section}
\definecolor{verdescuro}{RGB}{0,120,0}

\allowdisplaybreaks

\title{Non-equilibrium density fluctuations for \\stirring process with weak boundary dynamics}
\author{P. Birmpa \thanks{Department of Actuarial Mathematics \& Statistics, Heriot-Watt University and the Maxwell Institute for Mathematical Sciences, Mayfield Rd, Edinburgh EH9 3FD, UK, Email {\tt P.Birmpa@hw.ac.uk }} \and P. Gon\c calves\thanks{Instituto Superior T\'ecnico, Department of Mathematics, University of Lisbon, Av. Rovisco Pais 1, 1049-001, Lisbon. E-mail: {\tt pgoncalves@tecnico.ulisboa.pt}}
\and D. Tsagkarogiannis\thanks{Universit\`a dell'Aquila, DISIM, 67100, L'Aquila.  E-mail: {\tt dimitrios.tsagkarogiannis@univaq.it}}}

\date{}

\begin{document}
\maketitle
\abstract{We consider the one-dimensional stirring process on the segment 
$\{-N,\ldots,N\}$, coupled to boundary dynamics that inject particles from the
right reservoir and remove particles from the left reservoir, each acting on a window
of size $K$. We investigate the non-equilibrium fluctuations of the system, starting
from a product measure associated with a smooth initial profile. 
The limit Ornstein-Uhlenbeck process is given in terms of the Laplacian and gradient operators acting on test functions satisfying time-dependent Robin boundary conditions  that are obtained by linearizing the boundary conditions in the hydrodynamic equation.
A central ingredient
in our analysis is the derivation of sharp bounds on the space and space--time
$v$-functions of arbitrary degree for the centered occupation variables;  these  bounds were proved in the companion article \cite{birmpa2025nonequilibriumfluctuationsstirringprocess}.}

\section{Introduction}
The derivation of macroscopic equations of classical fluid mechanics from 
a microscopic system of interacting particles
is a long-standing problem in mathematical physics. In order to simplify this
challenge, a common approach is to consider random models from statistical mechanics, such as interacting
particle systems (IPS). These were introduced into the mathematics community by
Spitzer \cite{Spi}, although they had already been widely used in physics and
biophysics as microscopic stochastic models whose dynamics conserve certain quantities.
The basic assumption is that each particle behaves as a continuous-time random walk on a discretization of the macroscopic space using a scaling parameter $0<\epsilon:=N^{-1}$, which will eventually be sent to zero to connect microscopic and macroscopic descriptions.

{\textbf{The model under investigation.}} In this article we focus on the exclusion process, also called the stirring process: the particles move in the system, by exchanging their positions with their neighbors, after a mark of the Poisson process. We note that the number of particles is conserved under this dynamics.  We supplement this exchange dynamics with boundary dynamics that injects and removes particles in the system within a space window of size $K$. Note that the superposition of this dynamics with the exchange dynamics, destroys the conservation law of the system, though the number of particles is the quantity that we want to analyze. 
We also note  that this dynamics was first introduced in \cite{de2011current} but the boundary dynamics was slowed down with respect to  the exchange dynamics by a factor of $\eps$ and later the same dynamics slowed by a factor $\eps^\theta$, with $\theta>0$ 
were also explored in \cite{erignoux2020hydrodynamics,erignoux2019hydrodynamics}.
\bigskip
 
{\textbf{The hydrodynamic limit and the fluctuations around it.}} 
The hydrodynamic limit consists of describing the space-time evolution of the density of the particle in the system. To address this, one considers the empirical measure of the density defined in \eqref{MedEmp}, which consists of a random measure that associates the weight $\eps$ to every occupied site of the system and zero to the others. 
We denote by $\eta_t(x)$ the number of particles at site $x$ and time $t\eps^{-2}$ and if the system starts from a measure $\mu^\epsilon$ we denote 
the discrete profile by $\rho^\epsilon_t(x)=\mathbb{E}^{\epsilon}_{\mu^\epsilon}[\eta_{t}(x)]$. From Kolmogorov's forward equation, we get  $\partial_t \rho^\epsilon_t(x)=\mathbb E_{\mu^\epsilon}^\epsilon[L_\epsilon(\eta_{t}(x))]$, where $L_\eps$ is the generator of our Markov process.  
We observe that since the boundary dynamics defined above depends on the configuration $\eta$ at the boundary points within a window of size $K$, the previous discrete PDE for points in the boundary involves the expectation of terms given by products of $\eta$'s. Hence, in order to close it in terms of $\rho_t^\eps(x)$, we have to consider an auxiliary function that solves the same evolution equation as $\rho_t^\eps(x)$ but at the points in the boundary window we assume that the variables involved in the product are uncorrelated. This simple observation allows factorizing the expectation of those products into products of the corresponding $\rho$'s.   The advantage in considering the  function $\rho_\epsilon (x,t)$ win comparison to $\rho_t^\eps(x)$ is that the former solves a closed evolution equation while the latter depends on higher-order
moments.

  In a similar fashion, our goal in this article is to understand the behavior of fluctuations around the hydrodynamical profile $\rho_t$, solution of the hydrodynamic equation. To that end we consider the density fluctuation field which is the empirical measure properly centered and rescaled, see \eqref{density field} for the exact definition.
 In \cite{de2011current} the authors assumed that the function $\rho_{\epsilon}(\cdot,0)$ converges weakly as $\epsilon\to 0$ to a macroscopic profile  $u_0(\cdot )\in L^{\infty}([-1,1],[0,1])$ and deduced the hydrodynamic limit by showing that  $\rho_{\epsilon}(x,t)$ and $\rho_t$ are close in the sup norm, where the function $\rho_t$ is the solution of the heat equation with some boundary conditions, see Theorem 
 \ref{hydrodynamiclimit} for a precise statement. We note that  some years later, the same model but taking the boundary dynamics slowed down with respect to the exchange dynamics with a factor $\eps^\theta$, with $\theta>0$ was studied in \cite{erignoux2020hydrodynamics,erignoux2019hydrodynamics}. There, the authors derived the hydrodynamic limit when $\theta\geq 1$ using the entropy method developed by \cite{guo1988nonlinear} and when $\theta<1$ using the method of propagation of chaos, the same method used by \cite{de2011current} to establish the result in the case $\theta=1$. 
 In the same article, the authors were also able to show that $\rho_\eps(x,t)$ and $\rho_t^\eps(x)$ are close with respect to the sup-norm. This last result was obtained as a consequence of a much stronger result which was related to correlation estimates rather than the entropy method. 
Unfortunately, these estimates are not enough in order to prove the behaviour of the fluctuations around the hydrodynamic limit due to the different (weaker) scaling. A significant effort had to be invested in \cite{birmpa2025nonequilibriumfluctuationsstirringprocess} in order to improve these bounds and render them amenable to the case of the fluctuations.   We note that the relative entropy method of \cite{jara2018non} has been shown to be powerful for conservative particle systems, but its direct application to our model presents structural obstructions. The  scheme of \cite{jara2018non} requires local transition rates, since the entropy production is computed explicitly from the adjoint generator and relies on this locality. However, our  boundary rates have a range of order $K$ and  increase the degree of polynomials in $\eta$ at the boundary window, placing our model outside the class for which the  entropy estimates of \cite{jara2018non} are designed. It would be an interesting problem to check whether the methodology of \cite{jara2018non} could be extended to our case, and  we leave this as a future problem.

 In order to properly formalize our problem, we consider the density fluctuation field which is nothing but the integral of a test function $H$ with respect to the empirical measure properly centered and rescaled, see \eqref{density field} for the exact definition. 
We show that when the system starts from a product measure whose marginals are given in terms of a smooth profile, see {\textbf{Assumption 1}}, the sequence of density fields converges to the solution of an Ornstein-Uhlenbeck process given in \eqref{O.U.}.  This choice of initial condition is essential because it allows us to invoke the estimates from \cite{de2012truncated}, which are applicable only when the system starts from a product measure. If we had similar a priori estimates for the space of $v$-functions starting from general initial conditions, our results could also be less restrictive.
 The strategy of the proof is based on associating the Dynkin's martingales to the density fluctuation fields. Those will be close to the martingales appearing in the definition of the corresponding Ornstein-Uhlenbeck process, see \eqref{eq:martingale_problem}. Showing that these martingales are close is one of the biggest challenges of this work and it is based on refined bounds on space $v$-functions but also on space-time $v$-functions, as we have explained before, together with a proper choice of the space of test functions, that we explain now.
 
 \bigskip

{\textbf{The choice of the space of test functions.}}  Our density fluctuation field acts on functions that are $C^\infty$ in space. The regularity in time is needed so that we can apply Dynkin's formula to obtain a sequence of martingales associated to the density field and the regularity in space is needed since the hydrodynamic equation is the heat equation and we expect the limit field to satisfy the Ornstein-Uhlenbeck equation whose characteristic operators are the Laplacian and the gradient. Moreover, due to the action of the boundary dynamics on the conserved quantity, we obtain a contribution given in terms of products of occupation variables, see the last two lines of \eqref{intpartofmart_2}. The exchange dynamics also contributes with boundary terms at the end points $N$ and $-N$, and those correspond to the second line of \eqref{intpartofmart_2}. After centering these terms coming from the two contributions, we split them into degree one and higher-order degree. 

The degree-one boundary terms cancel
since the test functions satisfy  appropriate time-dependent Robin boundary
conditions; while the remaining
degree-one contributions vanish by our refined estimates on the space and space-time $v$-functions, as well as the  convergence of the discrete
hydrodynamic profile.

\bigskip

{\textbf{Future problems.}} We conclude by collecting a few interesting problems, closely related to the ones investigated
in this article, that remain open for further study.
A first natural question concerns the derivation of the non-equilibrium fluctuations for more general initial conditions as, for example, the non-equilibrium stationary state (NESS). It is  easy to check that the NESS is not of product form, and given the assumptions that we impose on our initial measure our result does not include the fluctuations from the NESS. A second direction is to analyze our results in the case the system is put in contact with slow/fast boundary dynamics and to explore what the dependence is on the strength of the dynamics in the $v$-functions estimates. A third line of investigation concerns the extension of our methods to other boundary
mechanisms. For instance, it would be interesting to treat boundary dynamics that drive the system to the heat equation with non-linear Robin boundary conditions, such as in the model recently  studied in \cite{landimmangisalvador}.

\begin{acknowledgements}
\small{The research of P.G. is partially funded by Fundação para a Ciência e Tecnologia (FCT), Portugal, through grant No. UID/4459/2025, as well as the ERC/FCT SAUL project. D.T. acknowledges financial support from the Italian Research Funding Agency
(MIUR) through the PRIN project “Emergence of condensation-like phenomena in interacting particle systems: kinetic and lattice models”, grant n. 202277WX43, as well as from Istituto Nazionale di Alta Matematica (INDAM).
The authors thank the warm hospitality of the University of Minho and Instituto Superior Técnico in Portugal, GSSI in Italy and the Maxwell institute for Mathematical Sciences, Edinburgh, UK, where parts of this work were developed. The authors are also indebted to Anna De Masi and Errico Presutti for many insightful and clarifying discussions.}
\end{acknowledgements}

%\section{Notations}

%\textcolor{red}{We should introduce the notation for $\lesssim$. We should put a note about constants, when  they depend on time and if we need to integrate in time or not... also fix notations for expectations: here we keep the same from De Masi's papers $\mathbb E_{\epsilon}[\cdot]$ and also for probabilities with respect to stirring}
\section{Statement of results}
\label{sec:1}

\subsection{The model}
\label{sec:model}

Let $N \in \mathbb{N}$ represent the scaling parameter, which we will consider in the limit  $N \to \infty$. Fix an arbitrary integer $1 \leq K \leq N$. We examine the stirring process on the discrete subspace $\Lambda_N := [-N, N] \cap \mathbb{Z}$. Particle configurations are represented by elements $\eta \in \Omega_N := \{0,1\}^{\Lambda_N}$, where $\eta(x) = 1$ indicates the presence of a particle at site $x$, and $\eta(x) = 0$ otherwise. The stirring process, incorporating births and deaths, is a Markov process $\{\eta_{t}\}_{t \geq 0}$ fully characterized by its infinitesimal generator given by 
      \begin{equation}
        \label{eq:gen_total}
L_\epsilon:=\epsilon^{-2}\big(L_0+\epsilon L_{b,-}+\epsilon L_{b,+}\big), \quad \epsilon\equiv \frac 1N,
     \end{equation}
where $L_0$ is the generator of the stirring process in $\Lambda_N$, namely
      \begin{equation}
        \label{eq:gen_stir}
 L_0 f(\eta):=\frac 12\sum_{x=-N}^{N-1} [f(\eta^{(x,x+1)})-f(\eta)]
     \end{equation}
and 
      \begin{equation}
        \label{eq:gen_bound}
L_{b,\pm} f(\eta):= \frac{j}{2}
\sum_{x\in I_\pm}D_{\pm}\eta(x) [f(\eta^{(x)})-f(\eta)\Big],
\end{equation}
with $j>0$, $I_+=[N-K+1,N]$, $I_-=[-N,-N+K-1]$,
\begin{equation}\label{eq:boundary_operators}
\begin{split}
& D_+u (x)= [1-u(x)]u(x+1)u(x+2)\cdots u(N), \; x\in I_+,
       \\
& D_-u (x)=  u(x)[1-u(x-1)][1-u(x-2)]\cdots[1-u(-N)], \;  x\in I_-\,.
\end{split}
   \end{equation}
   Note that  above  $D_-u(-N)=u(-N)$ and $D_+u(N)=1-u(N)$.
  
      Note that the generator includes a factor of $\epsilon^{-2}$, indicating that the process is considered on the diffusive time scale. In \eqref{eq:gen_bound}, $\eta^{x,x+1}$, with $-N \leq x \leq N-1$, represents the configuration obtained from $\eta$ by exchanging the occupation variables $\eta(x)$ and $\eta(x+1)$:
\[\eta^{(x,x+1)} (y) =
\begin{cases}
\eta(x+1),& y=x\;,\\
\eta(x), & y=x+1\;,\\
\eta(y), & y\neq x,x+1\;,\\
\end{cases}
\]
while $\eta^{(x)}$, for $x\in I_\pm $, is  the configuration obtained by flipping the
occupation variable $\eta(x)$ :
\[
\eta^{(x)}(y) =
\begin{cases}
\eta(y),& y\neq x\;, \\
1-\eta(y), & y=x\,.\\
\end{cases}
\]

The operator $N^2L_0$ is the generator of the stirring process, where neighboring particles exchange positions at rate $N^2/2$. The operator $NL_{b,\pm}$ represents the generator of a birth (resp. death) process: at rate $jN/2$, a particle is created (resp. destroyed) at the
rightmost empty site (resp. leftmost occupied site), see the figure below.

\begin{figure}[htb]
\begin{center}
\begin{tikzpicture}[thick, scale=0.85][h!]
\draw[latex-] (-6.5,0) -- (6.5,0) ;
\draw[-latex] (-6.5,0) -- (6.5,0) ;
\foreach \x in  {-6,-5,-4,-3,-2,-1,0,1,2,3,4,5,6}
\draw[shift={(\x,0)},color=black] (0pt,0pt) -- (0pt,-3pt) node[below] 
{};

%%%%%%%%%%%%%%%partÌculas azuis%%%%%%%%%%%%%%%%%%%%%
       \node[ball color=violet, shape=circle, minimum size=0.7cm] (U) at (-6.,1.2) {};
       \node[ball color=violet, shape=circle, minimum size=0.7cm]  at (-6.,0.4) {};
       \node[ball color=violet, shape=circle, minimum size=0.7cm] (R) at (6.,0.4) {};
         \node[ball color=violet, shape=circle, minimum size=0.7cm] (S) at (6.,1.2) {};
         \node[ball color=violet, shape=circle, minimum size=0.7cm] (T) at (6.,2.0) {};
        
%%%%%%%%%%%%%%%%partÌculas brancas%%%%%%%%%%%%%%%%%%%%%%%%         
        \node[shape=circle,minimum size=0.7cm] (Q) at (-6.,2) {};
        \node[shape=circle,minimum size=0.7cm] (P) at (-5.,1.2) {};
           \node[shape=circle,minimum size=0.7cm] (PP) at (-5.,0.4) {};
        \node[shape=circle,minimum size=0.7cm] (M) at (-6.,1.2) {};
        \node[shape=circle,minimum size=0.7cm] (N) at (-2.,1.2) {};
        \node[shape=circle,minimum size=0.7cm] (J) at (2.,0.2) {};
            \node[shape=circle,minimum size=0.7cm] (JJ) at (4,0.4) {};
    
%%%%%%%%%%%%%partÌculas pretas%%%%%%%%%%%%%%%%%%%%%%%%%%%%
\node[ball color=black!30!, shape=circle, minimum size=0.7cm] (C) at (-5,0.4) {};

\node[ball color=black!30!, shape=circle, minimum size=0.7cm] (C) at (-3,0.4) {};    
    
\node[ball color=black!30!, shape=circle, minimum size=0.7cm] (C) at (0,0.4) {};

\node[ball color=black!30!, shape=circle, minimum size=0.7cm] (D) at (1,0.4) {};

\node[ball color=black!30!, shape=circle, minimum size=0.7cm] (L) at (3.,0.4) {};

\node[ball color=black!30!, shape=circle, minimum size=0.7cm] (H) at (5.,0.4) {};

%%%%%%%%%%%%%%%%%%indices%%%%%%%%%%%%%%%%%%%%%%%%%%%%%%
\draw[] (-5,-0.2) node[below] {\footnotesize{$-N$}};
\draw[] (-4,-0.2) node[below] {\footnotesize{$\cdots$}};
\draw[] (-3,-0.2) node[below] {\footnotesize{$x-1$}};
\draw[] (-2,-0.2) node[below] {\footnotesize{$x$}};
%\draw[] (-5,-0.2) node[below] {\footnotesize{$1$}};

\draw[] (5,-0.2) node[below] {\footnotesize{$N$}};
\draw[] (4,-0.2) node[below] {\footnotesize{$N-1$}};
\draw[] (3,-0.2) node[below] {\footnotesize{$N-2$}};

\node[shape=circle,minimum size=0.7cm] (Q) at (-6.,2.0) {};
\node[shape=circle,minimum size=0.7cm] (W) at (3,2.0) {};

setinhas
\path [->] (PP) edge[bend right=60] node[above] {$\tfrac{jN}{2}$} (U);
\path [->] (T) edge[bend right=60] node[above] {$\tfrac{jN}{2}$} (JJ);
\path [->] (C) edge[bend left =60] node[] {$\times$} (D);

\path [->] (D) edge[bend left =60] node[above] {$\tfrac{N^2}{2}$} (J);
              
\node[shape=circle,minimum size=0.7cm] (K) at (0,0.4) {};
\node[shape=circle,minimum size=0.7cm] (G) at (2,0.4) {};
\node[shape=circle,minimum size=0.7cm] (E) at (-1.,0.5) {};
\node[shape=circle,minimum size=0.7cm] (EE) at (-3.,0.5) {};  
\end{tikzpicture}
\end{center}
\end{figure}
\subsection{Hydrodynamic limit}
\label{subsec:2}
In this section, we review the hydrodynamic limit of our model, as developed in \cite{de2011current,erignoux2020hydrodynamics}. The corresponding hydrodynamic equation is the linear heat equation on the ``macroscopic'' interval $(-1,1)$, subject to {Robin} boundary conditions at $\pm 1$. Before presenting the result, we introduce several key definitions. 

We denote by  $\| \cdot\|_{L^2  }$ the $\mathbb L^2$-norm  with respect to the Lebesgue measure in $[-1,1]$ and by $\langle \cdot,\cdot\rangle$ the corresponding inner-product.   For some fixed time horizon  $T>0$, we denote by $C^{m,n}([0, T] \times [-1,1])$ the set of functions defined on $[0, T] \times [-1,1] $ that are $m$ times differentiable on the first variable and $n$ times differentiable  on the second variable, {with} continuous derivatives. Now we define the Sobolev space $\mathcal H^1$ on $[-1,1]$. For that purpose, we define the semi inner-product $\langle \cdot, \cdot \rangle_{1}$ on the set $C^{\infty} ([-1,1])$ by $\langle G, H \rangle_{1} =\langle \partial_u  G \,, \partial_u  H \rangle$ 
and  the corresponding semi-norm is denoted by $\| \cdot \|_{1}$. 

\begin{definition}
\label{Def. Sobolev space}
The Sobolev space $\mathcal{H}^{1}$ on $[-1,1]$ is the Hilbert space defined as the completion of $C^\infty ([-1,1])$ for the norm 
$\| \cdot\|_{{\mathcal H}^1}^2 :=  \| \cdot \|_{L^2}^2  +  \| \cdot \|^2_{1}.$
Its elements coincide a.e. with continuous functions. The space $L^{2}(0,T;\mathcal{H}^{1})$ is the set of measurable functions $f:[0,T]\rightarrow  \mathcal{H}^{1}$ such that 
$\int^{T}_{0} \Vert f_{s} \Vert^{2}_{\mathcal{H}^{1}}ds< \infty. $
\end{definition}
\begin{definition}
	\label{def:weak_sol_ Robin}
	{Let $\rho_0:[-1,1]\rightarrow [0,1]$ be a measurable function.} We say that  $\rho:[0,T]\times[-1,1] \to [0,1]$ is a weak solution of the heat equation with Robin boundary conditions 
	\begin{equation}\label{eq:Robin_equation}
	\begin{cases}
	&\partial_{t}\rho(u,t)= \frac{1}{2} \partial^2_u\, {\rho}(u,t), \quad (t,u) \in [0,T]\times(-1,1),\\
	&\partial_u\rho(-1,t)=\widetilde D_{-}(\rho(-1,t))
,\quad t \in (0,T]\\ &\partial_u \rho(1,t)=\widetilde D_+(\rho(1,t))
	,\quad t \in (0,T] \\
	&\rho(\cdot,0)= \rho_0(\cdot),
	\end{cases}
	\end{equation}
where 
	\begin{equation}\label{tildeD}
	\widetilde  D_-(\rho)= j (1-(1-\rho)^K)\quad \textrm{and}\quad  \widetilde  D_+(\rho)= j(1-\rho^K)\end{equation}
	if the following two conditions hold: 
	\begin{enumerate}
		\item $\rho \in L^{2}(0,T;\mathcal{H}^{1})$, 
		
		\item $\rho$ satisfies the weak formulation:
		\begin{equation*}\label{eq:Robin_integral}
		\begin{split}
		&\langle \rho_{t},  f_{t}\rangle  -\langle \rho_0 , f_{0} \rangle - \int_0^t\langle \rho_{s},\Big(\frac 12 \partial^2_u + \partial_s\Big) f_{s} \, \rangle {ds}  
		+ \frac 12\int^{t}_{0}  \rho_{s}(1) \partial_u f_{s}(1)-\rho_{s}(-1)  \partial_u f_{s}(-1)  \, ds\\&
		 - \frac 12\int^{t}_{0}f_{s}(1)\widetilde D_+(\rho(1,s))
		ds
		+ \frac 12\int_0^t f_{s}(-1) \widetilde D_-(\rho(-1,s))ds=0,
		\end{split}   
		\end{equation*}
	for all $t\in [0,T]$, any function $f \in C^{1,2} ([0,T]\times[-1,1])$. 
	\end{enumerate}
	\end{definition}

	%\pat{For $K=2$ the equation \eqref{eq:Robin_equation} becomes
	%\begin{equation}\label{eq:Robin_equation_2}
	%\begin{cases}
	%&\partial_{t}\rho_{t}(u)= \frac{1}{2} \partial^2_u\, {\rho} _{t}(u), \quad (t,u) \in [0,T]\times(0,1),\\
	%&\partial_u\rho(t,-1)=j(1-(1-\rho(t,-1))^2)
%,\quad t \in [0,T]\\ &\partial_u \rho(t,1)=j(1-%\rho(t,1)^2),\quad t \in [0,T] \\
%	&\rho(0,\cdot)= \rho_0(\cdot),
%	\end{cases}
%	\end{equation}	
%	}
	\begin{remark}
	We warn the reader not to confuse the functions  $\widetilde D_\pm$ defined just above with the operators defined in \eqref{eq:boundary_operators}. 
	\end{remark}

 \begin{remark}
     We observe that the weak solution of the PDE given above is unique. We refer the reader to the proof in \cite{FGN_TAMS} for the case $K=1$ and \cite{erignoux2020hydrodynamics} for the other cases.
    \end{remark}

Observe that  the above equation coincides with the case $K=1$ of \cite{franco2019non} by taking there  $\alpha=0$ and $\beta=1$ and $j=1$ and in that case the Robin boundary conditions are linear.

Now we state the hydrodynamic limit of the process $\{\eta_{{t}}\}_{t\geq{0}}$. For that purpose,  let ${\mathcal  M}^+$ be the space of positive measures on $[-1,1]$ with total mass bounded by $2+\epsilon$ equipped with the weak topology. For any configuration  $\eta \in \Omega_{N}$ we define the empirical measure ${\pi^{\epsilon}(\eta,\cdot)\in{\mathcal  M}^+}$ on $[-1,1]$ as
\begin{equation}\label{MedEmp}
\pi^{\epsilon}(\eta, du)=\epsilon\sum _{x\in \Lambda_{N}}\eta(x)\delta_{\epsilon x}\left( du\right),
 \end{equation}
where $\delta_{a}$ is a Dirac mass on $a \in [-1,1]${. Given the trajectory $\{\eta_{{t}}\}_{t\geq{0}}$ of the \textit{accelerated} process, we further introduce 
$\pi^{\epsilon}_{t}(\eta,du):=\pi^{\epsilon}(\eta_{t}, du)$ the empirical measure at the macroscopic time $t$.
We denote by $\mathbb P _{\mu^\epsilon}$ the probability measure in the Skorohod space $\mathcal D([0,T], \Omega_N)$ induced by the  Markov process $\{\eta_{{{t}}}\}_{t\geq{0}}$ and the initial probability measure $\mu^\epsilon$ and  $\mathbb E^\epsilon _{\mu^\epsilon}$ denotes the expectation w.r.t. $\mathbb P_{\mu^\epsilon}$.  

\begin{theorem}
\label{th:hyd_ssep}
Let $\rho_0:[-1,1]\rightarrow[0,1]$ be a measurable function and let $\lbrace\mu^\epsilon\rbrace_{\epsilon}$ be a sequence of probability measures in $\Omega_{N}$ such that  for any continuous function $f:[-1,1]\rightarrow \mathbb{R}$  and every $\delta > 0$ 
\begin{equation}\label{eq:mea_associated}
  \lim _{\epsilon\to 0 } \mu^\epsilon \Big( \eta \in \Omega_{N} : \Big|\int_{-1}^1 f(u){\pi^\epsilon}(\eta,du) - \int_{-1}^1 f(u)\rho_{0}(u)du\Big|    > \delta \Big)= 0.\end{equation} Then, for any $t\in[0,T]$,
\begin{equation*}\label{limHidreform}
 \lim _{\epsilon\to 0} \mathbb P_{\mu^{\epsilon}}\Big( \Big|\int_{-1}^1 f(u) \pi_t^\epsilon(\eta, du) - \int_{-1}^1 f(u)\rho(u,t)du\Big|   > \delta \Big)= 0,
\end{equation*}
where  $\rho(\cdot,t)$ is the unique weak solution  of \eqref{eq:Robin_equation}.
\end{theorem}
Last theorem corresponds to Theorem 2.8 of \cite{erignoux2020hydrodynamics}. 
We observe that in last theorem the initial measure can be quite generic (not necessarily product) as long as it satisfies \eqref{eq:mea_associated} i.e. to be  associated to some profile. 
The next theorem, derived in  \cite{de2011current}, also provides the hydrodynamic limit for the stirring process with births and deaths. It is analogous to the previous one, but written in a different fashion. We decided to include both since some of these results will be needed in our proofs. For $x\in\Lambda_N$  and $t\geq 0$, let us denote 
\begin{equation}\label{eq:exp_eta}
\rho^\epsilon_t(x)=\mathbb{E}^{\epsilon}_{\mu^\epsilon}[\eta_{t}(x)].\end{equation}
From Kolmogorov's forward equation, the function $\rho^\epsilon_t(x)$ satisfies the evolution equation  $\partial_t \rho^\epsilon_t(x)=\mathbb E_{\mu^\epsilon}^\epsilon[L_\epsilon(\eta_{t}(x))]$, see \eqref{eq:kolm_disc_prof}.  Note that the boundary generators increase the degree of polynomial functions of $\eta$ therefore, we do not get a closed equation for $\partial_t \rho^\epsilon_t(x)$. To that end, we introduce the function $\rho_\epsilon (x,t)$ 
solution to
\begin{equation}\label{eq:linearized}
	\begin{cases}
	&\partial_t\rho_\epsilon(x,t)= \frac 12 \Delta_\epsilon\rho_\epsilon(x,t)+
\epsilon^{-1} \frac j2\Big(\mathbf 1_{x\in I_+} D_+\rho_\epsilon(x,t)
-\mathbf 1_{x\in I_-}  D_-\rho_\epsilon(x,t)\Big)\\
	&{\rho_\epsilon(x,0)=u_0(\epsilon x)},
	\end{cases}
	\end{equation}
where $\Delta_\epsilon=\epsilon^{-2}\Delta$, $\Delta$
the discrete Laplacian in $\Lambda_
N$ with reflecting boundary conditions:
    \begin{eqnarray}\label{eq:disc_laplacian}
&&\Delta u(x)=   u(x+1)+ u(x-1)-2u(x),\qquad |x|<\epsilon^{-1}
\\&&\Delta u(\pm N)= u(\pm(N-1))-u(\pm N).
         \label{def_laplacian}
        \end{eqnarray}

\begin{theorem}\cite [Theorem 1] {de2011current}\label{hydrodynamiclimit}
Suppose that the initial datum $\rho_{\epsilon}(\cdot,0)$ defined on $\Lambda_{N}$, and taking values in $[0,1]$ converges weakly as $\epsilon\to 0$ to $u_0(\cdot)\in L^{\infty}([-1,1],[0,1])$ in the sense
\begin{equation}\label{initial}
\lim_{\epsilon\to 0}\epsilon\sum_{x\in\Lambda_{N}}\rho_{\epsilon}(x,0)\phi(\epsilon x)=\int_{[-1,1]}u_0(r)\phi(r)dr,\quad\text{for every}\;\phi\in L^{\infty}([-1,1],\mathbb{R}).
\end{equation}
Then, there is $\rho(x,t)$, $x\in[-1,1]$, $t>0$ so that for any $t_1>t_0>0$
\begin{equation}\label{eq:hydro_conv}
\lim_{\epsilon\to 0}\sup_{x\in\Lambda_N}\sup_{t_0\leq t\leq t_1}|\rho_{\epsilon}(x,t)-\rho(\epsilon x,t)|=0,
\end{equation}
where the function $\rho(u,t)$  is the unique solution of the integral equation
\begin{eqnarray}\label{integralformofequation}
\rho(u,t)&=&\int_{[-1,1]}P_{t}(u,r')u_0(r')dr'+\frac{j}{2}\int_0^t P_{s}(u,1)(1-\rho(1,t-s)^{K}) \nonumber\\
&&-P_{s}(u,-1)(1-(1-\rho(-1,t-s))^{K})ds
\end{eqnarray}
and $P_{t}(u,r')$ is the density kernel of the semigroup (also denoted as $P_t$) with generator
$\Delta/2$, $\Delta$ the Laplacian in $[-1, 1]$ with reflecting, Neumann, boundary conditions.
\end{theorem}

\begin{remark}
    We  clarify now the relationship between the weak formulation given in Definition \ref{def:weak_sol_ Robin} and the integral formulation given in Theorem \ref{hydrodynamiclimit}. The latter is the mild formulation of the former. Indeed, the mild formulation of the hydrodynamic equation is given in \eqref{integralformofequation}.
The equivalence between the weak and the mild formulations follows standard arguments  based on testing the weak formulation against a regularized backward Neumann heat kernel, see e.g. Section 5 of \cite{erignoux2020hydrodynamics}. Starting from the weak solution, the bulk terms cancel because the heat kernel solves the backward heat equation, while the Neumann boundary condition on the kernel leaves only the nonlinear boundary fluxes. On the other hand, starting from the mild formulation and using the same heat-kernel identities we obtain the weak formulation. Hence, throughout the paper, the hydrodynamic profile appears either as the weak solution or the mild solution depending on the context.
\end{remark}
\begin{remark}
The assumptions in Theorems \ref{th:hyd_ssep} and \ref{hydrodynamiclimit} concern different initial measures. In the former, the initial condition is imposed on the random empirical measure associated with the microscopic configuration sampled from $\mu^\varepsilon$. In the latter the assumption is imposed on the deterministic initial datum of the auxiliary closed equation \eqref{eq:linearized}, namely $\rho^\varepsilon(\cdot,0)$. Under \textbf{Assumption 1} these two statements are  compatible, if the initial law is a product measure with marginals $\rho^\varepsilon _0(x)=u_0(\varepsilon x)$, then the empirical measure is associated with $u_0$; while  $\rho^\varepsilon(\cdot,0)$ converges to $u_0$, when $\epsilon\to 0$. 
\end{remark}
\begin{remark}\label{rem:smoothness}
We note that the solution $\rho(x,t)$ of \eqref{integralformofequation} is smooth. This has been observed in Remark 2.4 of \cite{de2011current}.
\end{remark}
Our goal in this article is to understand how is the behavior of the fluctuations around the hydrodynamical profile.

\subsection{Fluctuations}

In this subsection, we present our results regarding the non-equilibrium fluctuation of the stirring process with births and deaths, and  the convergence of the density fluctuation field given in Definition \ref{def:Density fluctuation field} to a mean zero  generalized Ornstein-Uhlenbeck process. These results are detailed in Theorem \ref{thm:non_eq_flu}. To begin, we recall from \cite{birmpa2017non} the definition of the space of test functions on which the functional associated with the system’s density fluctuations acts, as well as the semigroup associated with the hydrodynamic limit.
\subsubsection{The space of test functions}
{Let $\mathbb S=C^\infty([-1,1])$, endowed with its usual Fr\'echet
topology generated by the seminorms
\begin{equation}\label{semi-norm}
\|H\|_k=\sup_{u\in[-1,1]}|\partial_u^kH(u)|,
\qquad k\in\mathbb N\cup\{0\},
\end{equation}
and let $\mathbb S'$ denote its topological dual.} 
\begin{definition}\label{def1} 
Fix a solution $\rho(u,t)$  of the hydrodynamic equation \eqref{integralformofequation}   (which by Remark \ref{rem:smoothness} is a smooth function) and for $t\in(0,T]$, let $\mathbb D_t
$ denote the  subspace of functions in $C^{\infty}([-1,1])$ that satisfy  \begin{equation}\label{eq:bc_test_function}
\begin{split}
 &\partial_u H(-1)=\widetilde D'_-(\rho(-1,t))H(-1), \\&\partial_u H(1)=\widetilde D_+'(\rho(1,t))H(1),\end{split}
 \end{equation}
 where $\widetilde D_\pm$ were defined in \eqref{tildeD}.
\end{definition}
\begin{remark}
    The choice of test functions is dictated by the cancellation of the degree-one boundary terms in the Dynkin's martingale. The exchange dynamics produces boundary contributions involving discrete gradients  of the test function at boundary points, while the birth-death mechanism produces linear terms obtained by linearizing the nonlinear hydrodynamic profile. These terms cancel precisely when the test function satisfies the Robin boundary conditions given above. This will be clear ahead in Section \ref{sec: Auxiliary martingale}.
\end{remark}
   
{We now fix a hydrodynamic profile $\rho(u,s)$ solution of
\eqref{eq:Robin_equation}. In what follows, we will need to consider
the operator $T_{t,s}$, for $0\leq s\leq t\leq T$, associated with the
linear parabolic problem obtained by linearizing the boundary conditions
of the hydrodynamic equation. For $t\in[0,T]$ fixed and $H\in\mathbb S$,
define $T_{t,s}H$ as the solution at time $s$ of}
\begin{equation}\label{eq:Robin_equation_line}
	\begin{cases}
	&\partial_{s}T_{t,s}H (u)+ \frac{1}{2} \partial^2_u\, T_{t,s}H(u)=0, \quad (s,u) \in [0,t]\times(-1,1),\\
	&\partial_u T_{t,s}H(-1)=j K(1-\rho(-1,s))^{K-1}T_{t,s}H(-1)
,\quad s \in [0,t)\\ &\partial_u T_{t,s}H(1)=-jK\rho(1,s)^{K-1}T_{t,s}H(1)
	,\quad s\in [0,t) \\
	&T_{t,t}H(\cdot)= H(\cdot),
	\end{cases}
	\end{equation}
   Observe that above the boundary conditions can be rewritten as
	\begin{equation*}
\partial_u T_{t,s}H(-1)=\widetilde D'_-(\rho(-1,s))T_{t,s}H(-1)
,\quad \partial_u T_{t,s}H(1)=\widetilde D'_+(\rho(1,s))T_{t,s}H(1).
	\end{equation*}

   \begin{remark}
       Since the hydrodynamic profile $\rho$ is smooth  and bounded, the functions $s\mapsto \tilde D_-'(\rho(-1,s))$ and $s\mapsto \tilde  D_+'(\rho(1,s))$ are bounded and continuous. Therefore \eqref{eq:Robin_equation_line} is a linear non-autonomous parabolic problem with time-dependent Robin boundary conditions. Standard parabolic theory yields a unique solution which is smooth for $0\leq s<t$.
   \end{remark}

\subsubsection{The density fluctuation field}

We first introduce some notation. Recall  that we denote by $\mathbb P _{\mu^\epsilon}$ the probability measure in the Skorohod space denoted by $\mathcal D([0,T], \Omega_N)$ induced by the  Markov process $\{\eta_{t}\}_{t\geq{0}}$ and the initial probability measure $\mu^\epsilon$ and  $\mathbb E^\epsilon _{\mu^\epsilon}$ denotes the expectation with respect to $\mathbb P_{\mu^\epsilon}$ and that $
\rho^\epsilon_t(x)=\mathbb{E}^{\epsilon}_{\mu^\epsilon}[\eta_{t}(x)]$.

\begin{definition}[Density fluctuation field]\label{def:Density fluctuation field}
We define the density fluctuation field $  Y_\cdot^\epsilon$ as the time-trajectory of linear functionals acting on functions $H\in\mathbb S$ as
\begin{equation}\label{density field}
Y^\epsilon _t(H)\;=\;\sqrt \epsilon\sum_{x\in\Lambda_N}H(\epsilon x)
\Big(\eta_t(x)-\rho^\epsilon_t(x)\Big).
\end{equation}
\end{definition}
Note that the fluctuation field defined above is a random distribution
acting on smooth spatial test functions. However, in the martingale
formulation we shall also evaluate the field along time-dependent test
functions $H_s\in\mathbb D_s$.

\bigskip

\noindent
For each $\epsilon>0$, let $\mathbb{Q}^{\epsilon}_{\mu^{\epsilon}}$ be the probability measure on $\mathcal{D}([0,T],\mathbb{S}')$ induced by the
density fluctuation field $Y_{\cdot}^{\epsilon}$
and the measure $\mu^{\epsilon}$. 
For every $n$ positive integer,
\begin{equation}\label{eq:correlation}
C_t^{n,\epsilon}(x_{1},\dots,x_{n}):=\mathbb{E}^{\epsilon}_{\mu^\epsilon}\Big[\Big(\eta_t(x_1)-\rho^\epsilon_t(x_1)\Big)\dots\Big(\eta_t(x_n)-\rho^\epsilon_t(x_n)\Big)\Big]
\end{equation}

Moreover, we assume that   \begin{center}
{\bf Assumption 1.} The initial measure $\mu^{\epsilon}$ is a product measure that satisfies 
\begin{equation}\label{Ass1}
\rho^\epsilon_0(y)=\mathbb E^\eps_{\mu^{\epsilon}}[\eta(y)]=u_0(\epsilon y), \textrm{\;for all\;} y\in \Lambda_N.
\end{equation}
\end{center} 
Above  $u_0:[-1,1]\to[0,1]$ is the initial condition of  \eqref{eq:linearized} and we assume that it is  smooth. 
In contrast to the weaker assumption on the initial state given in \eqref{initial}, last assumption is essential for proving our results, i.e., Theorems \ref{thm:vestimate2} and \ref{thm:vestimate3}.
\begin{remark}\label{rem:v_1_null}
    From \textbf{Assumption 1}, we conclude that $\rho_0^\eps(y)=\rho_\eps(y,0)$ for all $y\in \Lambda_N.$
\end{remark}
One of the consequences of this assumption is that we are able to control  the discrete gradient of the solution of \eqref{eq:linearized} as stated in Corollary \ref{corollary: difference of rhos} and whose proof can be seen in \cite{birmpa2025nonequilibriumfluctuationsstirringprocess}.
To establish the  estimates for various space–time dependent quantities in this work, we must address two main difficulties with respect to previous works.  The first one is to show that those quantities are of order $\eps$ or $\eps^{1+\zeta}$, $\zeta>0$.  This is needed in order to have a closure of the equation and identify the limit.  The second is to retain a time exponent so that the resulting expressions remain integrable after integrating with respect to time. By imposing {\textbf{Assumption 1}}, the second difficulty can be resolved, providing a time exponent strictly greater than $-1$ and thus integrable close to 0. The same assumption is imposed in \cite{de1989weakly} and in Theorem 2 of \cite{de2011current} and in Theorem 2.1 in \cite{de2012non}.

    %\begin{lemma}\label{lem: difference of rhos}
%Let {$x\in \{-N,\dots, N-1\}$} and $0\leq t\leq T$. Under {\bf Assumption 1}, {for any $\eps<1$, } there exists a constant $C$ such that 
%\begin{eqnarray}
%|\rho_\epsilon(x,t)-\rho_\epsilon(x+1,t)|&\leq&{c_T\epsilon}+\epsilon^{-1}\frac{j}{2}\int_{0}^t d\lambda \frac{1}{\sqrt{\epsilon^{-2}\lambda}+1}\sum_{z\in I_\pm}G_{\epsilon^{-2}\lambda}(x,z).
%\end{eqnarray}
%\end{lemma}
%Above $G_{\epsilon^{-2}\lambda}(x,z)$ denotes the heat kernel given by   
%$G_t (x,y)=   \frac{e^{- (x-y)^2/2t}}{\sqrt{2\pi t}}.$
   For $H,G\in\mathbb S$  and $t\in[0,T]$,   we denote the norm  $\|\cdot\|_{L^2(\rho_t)}$ associated to the inner product defined by 
\begin{equation}\label{eq:norm}
\begin{split}
 \langle H,G\rangle_{L^2(\rho_t)}=&\int_{-1}^1\chi(\rho(u,t)) \partial_u H(u) \partial_u G(u)du \\
  +&\frac 12\widetilde D_+(\rho(1,t))H(1)G(1)+ \frac 12\widetilde D_-(\rho(-1,t))H(-1)G(-1),
\end{split}
\end{equation}
where $\rho_t(u)$ is the solution of the hydrodynamic equation and $\chi(u)=u(1-u)$, moreover, we define  the operator \[
\mathcal{P}_t(H)
:=
\|H\|_{L^2(\rho_t)}^2.
\]

Note that above the $\mathbb L^2$-norm is written with respect to a measure $\nu_t$ given by 
\[
\nu_t(du)
:=
\chi(\rho(u,t))\,du
+\frac{1}{2}\tilde D_{+}(\rho(1,t))\,\delta_{1}(du)
+\frac{1}{2}\tilde D_{-}(\rho(-1,t))\,\delta_{-1}(du).
\]

\begin{definition}\label{def:OU}
Let $\rho_t$ be the solution of the hydrodynamic equation \eqref{eq:Robin_equation}. We say that a stochastic process ${Y}$ in the space of continuous distributions $\mathcal{C}([0, T], \mathbb{S}')$   is the  solution of  the equation:
\begin{equation} \label{O.U.}
dY_t
=
\frac{1}{2}\partial_u^2 Y_t\,dt
+
\sqrt{\mathcal{P}_t}\,dW_t,
\end{equation}
where $ W_t$ is a space-time white noise of unit variance, if:
\begin{enumerate}[i)] 
    \item  for every
$H\in C^{1,\infty}([0,T]\times[-1,1])$ such that
$H_s\in\mathbb D_s$ for every $s\in(0,T)$, the stochastic processes
 $\{{M}_t(H), t\in[0, T]\}$ and $\{{N}_t(H), t\in[0, T]\}$ given by
    \begin{equation}\label{eq:martingale_problem}
        \begin{split}
        &{M}_t(H):={Y}_t(H_t)-{Y}_0(H_0)-\int_0^t{Y}_s\left(\partial_s H_s+\frac 12\partial_u^2 H_s\right)ds,
        \\&{N}_t(H):={M}_t(H)^2-\int_0^t\left\| H_s\right\|_{L^{2} (\rho_s)}^2 ds
        \end{split}
    \end{equation}
    are $\mathcal{F}_t$-martingales, where $\mathcal{F}_t:=\sigma\{{Y}_s(H): s\le t, H\in\mathbb{S}\}$,
    \item ${Y}_0$ is a mean-zero Gaussian field with covariance denoted  by
    \begin{equation}\label{eq:Y_covariance}
        \mathbb E\left[{Y}_0(H){Y}_0(G)\right]=:\sigma(H,G)
    \end{equation} for $H,G\in\mathbb S$.
\end{enumerate}

\end{definition}

Now we comment about the uniqueness in law of the solution of the martingale problem given above. From the proof of Theorem \ref{thm:non_eq_flu}, which involves a compactness argument, we will prove the existence of a process that solves that martingale problem. In order to conclude the full convergence result we need to establish the uniqueness in law of such process. We do not present the complete proof of the uniqueness but we outline the strategy to be pursued, which follows similar steps to those of the proof of Proposition 5.1 of \cite{franco2019non}.  The idea is to frame the PDE \eqref{eq:Robin_equation_line} as a regular problem in the classical Sturm-Liouville theory. This is possible given the fact that the solution of the hydrodynamic equation is a continuous function of time, as discussed in Section  2.4, formula (12), in \cite{de2011current}. We can then obtain a similar result to that of Lemma  5.2 of \cite{franco2019non} which is a consequence of Proposition 3.1 and Corollaries 3.2 and 3.3 of \cite{franco2019non}. We leave the details to the reader.

\begin{theorem}[Non-equilibrium fluctuations: the Ornstein-Uhlenbeck limit]\label{thm:non_eq_flu}
\quad

The sequence $\{\mathbb{Q}^{\epsilon}_{\mu^{\epsilon}} \}_{\epsilon}$ converges, as $\epsilon\to 0$, in $\mathcal D([0,T],\mathbb S')$, to a generalized Ornstein-Uhlenbeck process  solution of \eqref{O.U.}. As a consequence, the covariance of the limit field ${Y}_t$ is given on $H,G\in{\mathbb  S}$ by
\begin{equation}\label{covariance non eq limit field}
 E\,[{Y}_t(H){Y}_s(G)]\;=\;\int_{-1}^1\chi(u_0(u)) T_{t,0}H(u)T_{s,0}G(u) du+\int_0^s\langle  T_{t,r} H,  T_{s,r}G\rangle_{L^2(\rho_r)}dr,
\end{equation} for $ s\leq t$,
where   $T_{t,r}H$ (resp. $T_{s,r}G$) solves \eqref{eq:Robin_equation_line} with terminal condition $H$ (resp. $G$),  $\chi(u)=u(1-u)$ and 
 for $r>0$  and  $H,G\in\mathbb S$ the inner product $\langle H_r,G_r\rangle_{ L^2(\rho_r)}$ was defined above and  $\rho(u,r)$ is the solution of the hydrodynamic equation.
\end{theorem}

\subsection{Estimates on space-time $v$-functions}

The proof of Theorem \ref{thm:non_eq_flu} relies on the analysis of the auxiliary martingales that converge as $N\to+\infty$, to those defined in \eqref{eq:martingale_problem}. In this context, the space-time  $v$-function defined below naturally arises as it controls  the difference between  the expectation of products of $\eta_t(x)$ and
$\rho_\eps(x,t)$ at different times and spatial points. Obtaining appropriate estimates for this function is therefore essential for our applications.  First we introduce the space $v$-function.
\begin{definition}\label{definitionvestimate}We fix an initial product measure $\mu^{\epsilon}$ on  $\Lambda_N$.
%shorthand by $\E_\eps$ the expectation for the
%process with generator $L_\eps=\eps^{-2}L_0 +\eps^{-1} L_b$ which
%starts from  $\eta$ and write $\rho_\eps(x,t)$ for the solution of
%\eqref{eq:linearized} with initial condition $\eta$. 
%Recall  that the
%$v$-functions are defined in \eqref{pro.2.3.1} and, for brevity, we
%shall write $v(\und x,t)\equiv v^\eps(\underline x,t|\eta)$.
Suppose that the process $\{\eta_{t}\}_{t\geq 0}$ starts with $\mu_{\epsilon}$ and $\rho_{\epsilon}(t,x) $ is the solution of \eqref{eq:linearized}. We define the space  $v$-function as:
\begin{equation}\label{def: v function original}
v_{n}^{\epsilon}(\underline{x},t|\mu^{\epsilon}):=\mathbb{E}^{\epsilon}_{\mu^\epsilon}\left[\prod_{i=1}^{n}\left(\eta_{t}(x_i)-\rho_{\epsilon}(x_{i},t)\right)\right],\quad \underline{x}\in\Lambda_{N}^{n,\neq},n\geq1
\end{equation}
where $\Lambda_{N}^{n,\neq}$ is the set of all sequences $(x_1,\dots,x_n)\in \Lambda_{N}^{n}$ with $x_i\neq x_j$. For brevity, we
shall write $v_{n}^{\epsilon}(\underline{x},t)\equiv v_{n}^{\epsilon}(\underline{x},t|\mu^{\epsilon})$. When $\underline{x}=x$ then 
\begin{equation} \label{eq:important_1}
v_{1}^{\epsilon}(x,t):=\rho_t^{\epsilon}(x)-\rho_{\epsilon}(x,t).
\end{equation}
and when $\underline{x}=\und\emptyset$ is the empty configuration (i.e., no particles in the system), then 
\begin{equation}\label{v:emptyset}
v_0^{\eps}(\und \emptyset,t)=1.
\end{equation}
\end{definition}
We observe that as a consequence of  Remark \ref{rem:v_1_null} we have  $v^\eps_1(x,0)\equiv 0$.

In Theorem 2.1 of \cite{de2012truncated} the authors obtained estimates for 
$v_{n}^{\epsilon}(x,t)$   which depend on the number of particles $n$. Those estimates proved insufficient for our purposes-specifically, for achieving closure of the auxiliary martingales presented in Section \ref{sec: Auxiliary martingale}. Therefore, refined estimates were provided in the companion article \cite{birmpa2025nonequilibriumfluctuationsstirringprocess}, and we recall them here for completeness.
Before we proceed, we note that when $K=1$ the result below is not necessary for the proof of non-equilibrium fluctuations as it was given in \cite{franco2019non,gonccalves2020non}. Nevertheless, when $2\leq K\leq N$ we need a much more refined analysis not only for the closure of the martingale but also to show tightness of the sequence of the fluctuation fields. 

We note however that the estimates for the space $v$-function given in the next theorem continue to be insufficient  when $n=1,2$ as we do not get a bound of order $\epsilon$ contrary to the case $K=1$, where in  Proposition 8.1 of  \cite{franco2019non} it was derived a bound of order $\epsilon$ for the two point correlation function. That bound was sufficient to derive the non-equilibrium fluctuations of the model in the case $K=1$. Consequently, the next result and all results that follow, are stated for $K\geq 2$. For a proof we refer the reader to \cite{birmpa2025nonequilibriumfluctuationsstirringprocess}.
\begin{theorem} [Space correlation $v$-estimate, \cite{birmpa2025nonequilibriumfluctuationsstirringprocess}]
	\label{thm:vestimate2}
Let $K\geq 2$ be a fixed  integer number, and $\zeta>0$. For any initial product measure $\mu^{\epsilon}$ satisfying \eqref{Ass1},  and any time  $0<t\leq T$
{\begin{equation}\label{newvestimate}
\sup_{\underline{x}\in\Lambda_{N}^{n,\neq}}|v_{n}^{\epsilon}(\underline{x},t|\mu^{\epsilon})|\lesssim   \epsilon^{1-2\zeta}\mathbf{1}_{n=1,2}+\epsilon\mathbf{1}_{3\leq n\leq K}+\epsilon^{1+\zeta}\mathbf{1}_{n\geq K+1}.\end{equation}
}
\end{theorem}
\begin{remark}\label{rem:ind_time_space_v}
Above and in the following, $\lesssim$ denotes an inequality that is correct up to a
multiplicative universal constant, that is, independent of $N, \epsilon, t$. 
\end{remark}

As a consequence:
\begin{corollary}[\cite{birmpa2025nonequilibriumfluctuationsstirringprocess}]\label{cor: thm for space v_initial}
Let  $K\geq 2$ be a fixed  integer number  and $\zeta>0$. For any initial product measure $\mu^{\epsilon}$ satisfying \eqref{Ass1},  and time  $0<t\leq T$
\begin{equation}\label{d_initial}
\sup_{\underline{x}\in\Lambda_{N}^{n,\neq}}|v^{\epsilon}_{n}( \underline{x},t|\mu^{\epsilon})|-v^{\epsilon}_{n}( \underline{x}+\und e_1,t|\mu^{\epsilon})|\lesssim
\begin{cases}
\epsilon,&\;1\leq n\leq K\\
 \epsilon^{1+\zeta}, &n\geq K+1,
\end{cases}
\end{equation}
where $\und e_1$ being the unit vector in the positive 1st-direction.
 \end{corollary}

 In our case, since the boundary dynamics increases the degree of polynomial functions of $\eta$ the analysis is much more complicated and we cannot get the bound $\epsilon$ as desired. To overcome that issue, we instead gain a better rate of convergence to $0$ by analyzing space-time $v$-functions  that we define as follows. \begin{definition}\label{def:space time correlations}
Let $\underline{x}=(x_1,\dots,x_n)\in \Lambda_{N}^{n,\neq}$ and $\underline{y}=(y_1,\dots,y_m)\in \Lambda_{N}^{m,\neq}$ with $n,m\geq 1$. For any initial product measure $\mu^{\epsilon}$ and times $0<s<r$, the space-time correlation $v$-function is given by 
\begin{equation}\label{def:space time corr fun}
v^{\epsilon,m,n}_{s,r}(\underline{y}, \underline{x}|\mu^{\epsilon}):=\mathbb{E}^{\epsilon}_{\mu^\epsilon}\left[\prod_{j=1}^{m}\bar\eta_s(y_j)\prod_{i=1}^{n}\bar\eta_r(x_i)\right]
\end{equation}
where  $\bar\eta_t(x):=\eta_t(x)-\rho_\eps(x,t)$, $t\geq 0$.
\end{definition}
 The subscript indicates the times $s,r$ involved in the space-time correlation $v$-function. The inputs $(\underline{y},\underline{x})$ are the positions of the $m$ particles at time $s$, denoted by  $\underline{y}=(y_1,\dots,y_m)$ and the positions of the $n$ particles at time $r$, denoted by  $\underline{x}=(x_1,\dots,x_n)$.
The superscript reflects the dependence on $\epsilon$, which may sometimes be omitted, and the number of particles at time $s$ and at time $r$, i.e., $m$ and $n$ respectively. 
In Theorem \ref{thm:vestimate3}, we establish a space-time correlation estimate. This theorem demonstrates that the $\mathbb{L}^2$-norms of  \eqref{eq:diff_eta}–\eqref{degree_four} are bounded by quantities that vanish as $\epsilon\to0$. In Section \ref{sec: Auxiliary martingale}, we revisit the estimation of these terms, providing detailed calculations to complete the analysis.

\begin{theorem}[Space-time correlation $v$-estimate]
	\label{thm:vestimate3}
Let  $K\geq 2$ be a fixed  integer number, $n,m\geq 1$ and $\zeta>0$. For any initial product measure $\mu^{\epsilon}$ satisfying \eqref{Ass1}, and times  $0<s<r\leq T$, there exists $\tilde\Delta>0$ (independent of $\eps$), such that when $r-s\leq \tilde\Delta$
\begin{equation}\label{bound:space_time_v_estimate1}
\sup_{\underline{y}\in\Lambda_{N}^{m,\neq}}\sup_{\underline{x}\in\Lambda_{N}^{n,\neq}}|v^{\epsilon,m,n}_{s,r}(\underline{y}, \underline{x}|\mu^{\epsilon})|\lesssim 
\begin{cases}
 \epsilon(r-s)^{-1/2},&n+m=2,3\\[2mm]
 \epsilon^{1+\zeta}(r-s)^{-1+\zeta},&n+m\geq 4,\;n=m\\[2mm]
\epsilon^{1+\zeta}(r-s)^{-/2},&n+m\geq 4,\;n\neq m.
\end{cases}
\end{equation}
When $r-s> \tilde\Delta$,   
\begin{equation}\label{bound:space_time_v_estimate2}
\sup_{\underline{y}\in\Lambda_{N}^{m,\neq}}\sup_{\underline{x}\in\Lambda_{N}^{n,\neq}}|v^{\epsilon,m,n}_{s,r}(\underline{y}, \underline{x}|\mu^{\epsilon})|\lesssim
\begin{cases}
 \epsilon,&n+m=2,3\\[2mm]
 \epsilon^{1+\zeta},&n+m\geq 4.
\end{cases}
\end{equation}
\end{theorem}

From the previous result it is easy to obtain the next result.

\begin{corollary}[\cite{birmpa2025nonequilibriumfluctuationsstirringprocess}, Corollary~6.5]
\label{cor:vestimate3_difference}
Let $K\geq2$ be a fixed integer and let
$\zeta>0$ be sufficiently small.
For any initial product measure $\mu^\epsilon$
satisfying \eqref{Ass1}, there exist $N_0\geq K+3$ and
$\tilde\Delta\in(0,1)$, independent of $\epsilon$,
such that the following estimates hold uniformly
for $0<s<r\leq T$. 
\smallskip

\noindent  For $1\leq n,m\leq N_0-1$, if
$0<r-s\leq\tilde\Delta$, then
\begin{equation}
\label{bound:space_time_difference_v_estimate1}
\begin{aligned}
&\sup_{\underline y\in\Lambda_N^{m,\neq}}
\sup_{\substack{
\underline x\in\Lambda_N^{n,\neq}\\
\underline x+\underline e_1\in\Lambda_N^{n,\neq}
}}
\Big|
v_{s,r}^{\epsilon,m,n}
(\underline y,\underline x\mid\mu^\epsilon)
-
v_{s,r}^{\epsilon,m,n}
(\underline y,\underline x+\underline e_1
\mid\mu^\epsilon)
\Big|
\lesssim
\frac{\epsilon^{2-4\zeta}}
{(r-s)^{1-\zeta}}.
\end{aligned}
\end{equation}

\noindent If $r-s>\tilde\Delta$, then
\begin{equation}
\label{bound:space_time_difference_v_estimate2}
\begin{aligned}
&\sup_{\underline y\in\Lambda_N^{m,\neq}}
\sup_{\substack{
\underline x\in\Lambda_N^{n,\neq}\\
\underline x+\underline e_1\in\Lambda_N^{n,\neq}
}}
\Big|
v_{s,r}^{\epsilon,m,n}
(\underline y,\underline x\mid\mu^\epsilon)
-
v_{s,r}^{\epsilon,m,n}
(\underline y,\underline x+\underline e_1
\mid\mu^\epsilon)
\Big|
\lesssim\epsilon^{2-4\zeta}.
\end{aligned}
\end{equation}

\noindent Moreover, for $n\geq N_0$ and $m\geq1$,
we have
\begin{equation}
\label{bound:space_time_difference_v_high_degree}
\begin{aligned}
&\sup_{\underline y\in\Lambda_N^{m,\neq}}
\sup_{\substack{
\underline x\in\Lambda_N^{n,\neq}\\
\underline x+\underline e_1\in\Lambda_N^{n,\neq}
}}
\Big|
v_{s,r}^{\epsilon,m,n}
(\underline y,\underline x\mid\mu^\epsilon)
-
v_{s,r}^{\epsilon,m,n}
(\underline y,\underline x+\underline e_1
\mid\mu^\epsilon)
\Big|\\
&\qquad\lesssim
\frac{\epsilon^{2-4\zeta}}{(r-s)^{1/2}}
\mathbf{1}_{\{r-s\leq\epsilon^{\beta^*}\}}
+
\epsilon^{1+\zeta}
\mathbf{1}_{\{r-s>\epsilon^{\beta^*}\}}.
\end{aligned}
\end{equation}
Here $\underline e_1$ denotes the unit vector
in the positive first-coordinate direction.
\end{corollary}

\section{Proof of Theorem \ref{thm:non_eq_flu}}

\subsection{Auxiliary martingales}\label{sec: Auxiliary martingale}
{Fix $t\in[0,T]$ and let  $H$ be a function in $C^{1,2}([0,T]\times[-1,1])$ such that
$H_s\in\mathbb D_s$ for every $s\in(0,t]$.} For now we do not assume anything else on this function, but below we will see that we will need to impose other conditions on $H$ in order to obtain  that the limit field satisfies 
 \begin{equation}\label{characterization}
{Y}_t(H)\;=\; {Y}_0(T_{t,0} H)+ W_t(H),
\end{equation}
for any $H\in\mathbb S$, where 
 $T_{t,s}$ is the solution of \eqref{eq:Robin_equation_line} and  $W_t(H)$ is a mean zero Gaussian variable with  variance  
\begin{equation}\label{eq212}
 \int_0^t\| T_{t,s}  H\|^2_{\mathbb L^2(\rho_s)}ds.
\end{equation}
From  \cite[p. 330]{kipnis2013scaling} we know  that
\begin{align}
M_t^\epsilon(H)& \;:=\;\ Y_t^\epsilon (H_t)- Y^\epsilon _0(H_0)-\int_{0}^t \Lambda_s^\epsilon(H_s)\,ds\,,\label{martingaleM} \\
N_t^\epsilon(H)& \;:=\;(M_t^\epsilon (H))^2-\int_{0}^t \Gamma^\epsilon_s(H_s)\,ds\label{quadratic}
\end{align}
are martingales with respect to  $\mathcal F_t$, where
\begin{equation*}
\begin{split}
& \Lambda_s^\epsilon(H_s)\;:=\; (\partial_s+L_\epsilon ) Y_s^\epsilon (H_s)\,,\\
& \Gamma^\epsilon(H_s)\;:=\; L_\epsilon  Y_s^\epsilon(H_s)^2-2 Y_s^\epsilon (H_s)L_\epsilon  Y_s^\epsilon(H_s)\,.
\end{split}
\end{equation*}
From the computations in Appendix \ref{ap:action_gen} we see that 
\begin{equation}\label{intpartofmart_1}
\begin{split}
\Lambda_s^\epsilon(H_s)\;&=\;  Y_s^\epsilon\Big(\frac 12\Delta_\epsilon H_s+\partial_sH_s\Big)\\
&+\frac{j}{2\sqrt \epsilon} \sum_{x\in I_+ } H_s(\epsilon x) \Big(D_+  \eta_{s} (x)-\mathbb E^\epsilon_{\mu^\eps}[D_+\eta_{s}(x)]\Big)\\
&-\frac{j}{2\sqrt \epsilon} \sum_{x\in I_- } H_s(\epsilon x) \Big(D_-  \eta_{s} (x)-\mathbb E^\epsilon_{\mu^\eps}[D_-\eta_{s}(x)]\Big).
\end{split}
\end{equation}
Note that the discrete laplacian has been defined in \eqref{eq:disc_laplacian}.
Now we isolate the boundary contribution from $Y_s^\epsilon\Big(\frac 12\Delta_\epsilon H_s+\partial_sH_s\Big)$ to  obtain:
\begin{equation}\label{intpartofmart_2}
\begin{split}
\Lambda_s^\epsilon(H_s)\;&=\; \sqrt \epsilon \sum_{x=-(N-1)}^{N-1}\frac 12\Delta_\epsilon H_s (\epsilon x)\Big(\eta_{s}(x)-\rho^\epsilon_{s}(x)\Big)+Y_s^\epsilon(\partial_sH_s)\\
&+\frac{1}{2\sqrt \epsilon}\nabla_\epsilon^+ H_s(-1) \bar{ \eta}_{s} (-N)\textcolor{red}{+}\frac{1}{2\sqrt \epsilon}\nabla_\epsilon^- H_s(1) \bar{ \eta}_{s} (N)\\
&+\frac{j}{2\sqrt \epsilon} \sum_{x\in I_+ } H_s(\epsilon x) \Big(D_+  \eta_{s} (x)-\mathbb E^\epsilon_{\mu^\eps}[D_+\eta_{s}(x)]\Big)\\
&-\frac{j}{2\sqrt \epsilon}\sum_{x\in I_- } H_s(\epsilon x) \Big(D_-  \eta_{s} (x)-\mathbb E^\epsilon_{\mu^\eps}[D_-\eta_{s}(x)]\Big).
\end{split}
\end{equation}
Above the notation $\nabla_\epsilon ^\pm $ denotes the discrete gradient: \[
\begin{aligned}
\varepsilon \nabla_{\varepsilon}^{-} H(1,t)
&:= H(\varepsilon(N-1),t) - H(1,t), \\
\varepsilon \nabla_{\varepsilon}^{+} H(-1,t)
&:= H(\varepsilon(-N+1),t) - H(-1,t).
\end{aligned}
\]
Note that the first term on the right-hand side of \eqref{intpartofmart_2} will give rise macroscopically, when sending $\epsilon \to 0$,  to the term $Y_s (\partial_u^2 H_s)$, since the discrete Laplacian will converge to the continuous one. {To prove that it is enough to apply Taylor expansion  to the function $H_s$ and use a Cauchy-Schwarz inequality and the fact that the variables are bounded. }
Now we analyze the third line of the last display. The analysis of the fourth line is completely analogous.
Observe that by a Taylor expansion of $H_s$ around $1$, the last line of last display is equal to 
\begin{equation}\label{eq:boundary_term_final}
\frac{j}{2\sqrt \epsilon} H_s(1)\sum_{x\in I_+ }  \Big(D_+  \eta_{s} (x)-\mathbb E^\epsilon_{\mu^\eps}[D_+\eta_{s}(x)]\Big)+  O_{H,K}(\sqrt \epsilon).
\end{equation}

At this point, we specify $K=2$ to make the presentation simpler. The computations for general $K$ are given in Appendix  \ref{ap:boundary_terms}.
In this setting, note that  $I_+=\{N-1,N\}$ and  $D_+\eta(N)=(1-\eta(N))$ and $D_+\eta(N-1)=(1-\eta(N-1))\eta(N)$. Therefore the leading term of \eqref{eq:boundary_term_final} writes as 
\begin{equation}\label{eq:dynkin_k=2}
\frac{j}{2\sqrt \epsilon} H_s(1)\Big(\mathbb E^\epsilon_{\mu^\eps}[\eta_s(N)\eta_s(N-1)]- \eta_s(N)\eta_s(N-1)\Big).
\end{equation}
At this point we observe that in order to continue we will rewrite the martingales in terms of the variables centered with respect to $\rho_\epsilon$ as they appear in the  $v$-functions estimates.  
To that end, note that
\begin{equation*}
\eta_s(N)\eta_s(N-1)=\bar \eta_s(N)\bar \eta_s(N-1)+\rho_\epsilon(N,s)\eta_s(N-1)+\rho_\epsilon (N-1,s)\eta_s(N)-\rho_\epsilon(N,s)\rho_\epsilon(N-1,s),
\end{equation*}
where $\bar\eta_s(x)=\eta_s(x)-\rho_\epsilon ( x,s)$. 
Recall that 
\begin{equation*}
\begin{split}
v_2^\epsilon(N-1,N,s|\mu^\epsilon)=\mathbb E^\epsilon_{\mu^\eps}[(\eta_s(N-1)-\rho_\epsilon(N-1,s))(\eta_s(N)-\rho_\epsilon(N,s))].
\end{split}
\end{equation*}
Therefore,  
\begin{equation*}
\begin{split}
\mathbb E^\epsilon_{\mu^\eps}[\eta_s(N)\eta_s(N-1)]- \eta_s(N)\eta_s(N-1)&=v_2^\epsilon(N-1,N,s|\mu^\epsilon)-  \bar\eta_s(N)\bar\eta_s(N-1)\\&+\rho_\epsilon(N,s)\Big(\rho^\epsilon_s(N-1)-\eta_s(N-1)\Big)\\&+\rho_\epsilon(N-1,s)\Big(\rho^\epsilon_s(N)-\eta_s(N)\Big).
\end{split}
\end{equation*}
Simple computations show that the  last display is equal to 
\begin{equation*}
\begin{split}
&v_2^\epsilon(N-1,N,s|\mu^\epsilon)-  \bar\eta_s(N)\bar\eta_s(N-1)+\rho_\epsilon(N,s)v_1^\epsilon(N-1,s|\mu^\epsilon)+\rho_\epsilon(N-1,s)v_1^\epsilon(N,s|\mu^\epsilon)\\&- \rho_\epsilon(N,s)[\bar\eta_s(N-1)-\bar\eta_s(N)]
-(\rho_\epsilon(N-1,s)-\rho_\epsilon(N,s))\bar\eta_s(N)\\&+2\Big(\rho(1,s)-\rho_\epsilon(N,s)\Big)\bar\eta_s(N)-2\rho(1,s)\bar\eta_s(N).
\end{split}
\end{equation*}

From this we can write  the time  integral in \eqref{martingaleM} as 
\begin{align}
&\int_{0}^t  Y_s^\epsilon(\frac 12\partial_u^2 H_s+\partial_sH_s){+}\frac{1}{2\sqrt \epsilon}\nabla_\epsilon^- H_s(1) \bar{ \eta}_{s} (N)-2\frac{j}{2\sqrt \epsilon} H_s(1) \rho(1,s)\bar\eta_s(N)ds\label{eq:bc_dy}\\
+&\int_{0}^t \frac{j}{2\sqrt \epsilon} H_s(1) v_2^\epsilon(N-1,N,s|\mu^\epsilon)ds\label{eq:v_2_term}\\ 
+&\int_{0}^t \frac{j}{2\sqrt \epsilon} H_s(1) \Big\{\rho_\epsilon(N,s)v_1^\epsilon(N-1,s|\mu^\epsilon)+\rho_\epsilon(N-1,s)v_1^\epsilon(N,s|\mu^\epsilon)\Big\}ds \label{eq:v_1_term}\\
-&\int_{0}^t \frac{j}{2\sqrt \epsilon} H_s(1) \rho_\epsilon(N,s)[\bar\eta_s(N-1)-\bar\eta_s(N)] ds\label{eq:diff_eta}\\
-&\int_{0}^t \frac{j}{2\sqrt \epsilon} H_s(1)
(\rho_\epsilon(N-1,s)-\rho_\epsilon(N,s))\bar\eta_s(N)ds\label{diff_of_rho_epsilon}\\
+&\int_{0}^t \frac{j}{2\sqrt \epsilon} H_s(1)\Big\{2\Big(\rho(1,s)-\rho_\epsilon(N,s)\Big)\bar\eta_s(N)\Big\} ds \label{degree_one_plus_hydro_limit}\\
-&\int_{0}^t \frac{j}{2\sqrt \epsilon} H_s(1) \bar\eta_s(N)\bar\eta_s(N-1)ds\label{degree_four}
\end{align}
plus similar terms regarding the left boundary as well as terms that vanish as $\eps\to 0$. Our goal is now to show that, as $\epsilon \to0$, all terms--except for the leftmost one in the first line of the last display--vanish in the $\mathbb L^2$-norm.  We note that our approach consists in taking the terms above that involve $v$-functions and estimating them directly with the Cauchy-Schwarz inequality. For the  remaining terms we use $\mathbb L^2$  bounds and estimates on space-time $v$-functions. Before that, note that  from a Taylor expansion of the function $H$ at $x=1$, we can write the two rightmost terms in \eqref{eq:bc_dy} as  
\begin{align}\begin{split}\label{eq:space_kill}
&\int_{0}^t  \frac{1}{2\sqrt \epsilon}\nabla_\epsilon^- H_s(1) \bar{ \eta}_{s} (N)-2\frac{j}{2\sqrt \epsilon} H_s(1) \rho(1,s)\bar\eta_s(N)ds\\=&\int_{0}^t  \frac{1}{2\sqrt \epsilon}\Big\{{-}\partial_u H_s(1)+\frac{\epsilon}{2} \partial_u ^2H_s(\xi)\Big\} \bar{ \eta}_{s} (N)-2\frac{j}{2\sqrt \epsilon} H_s(1) \rho(1,s)\bar\eta_s(N)ds
\end{split}\end{align}
where $\xi\in(1-\epsilon, 1)$. Note that  $H_s\in\mathbb D_s$, i.e., in particular
 $\partial_u H_s(1)=-2j\rho(1,s)H_s(1)$ for any $s\in[0,t]$. The last identity is simply \eqref{eq:bc_test_function} for the particular case $K=2$. Then the last display is equal to 
\begin{align}\begin{split}
{-}\int_{0}^t  \frac{\sqrt \epsilon }{2} \partial_u^2H_s(\xi) \bar{ \eta}_{s} (N)ds.
\end{split}\end{align}
Since the occupation variables are bounded, 
the $\mathbb L^2$-norm of last display can be bounded by a term of order $O(\sqrt \epsilon)$, and it vanishes as $\epsilon \to0$.

\subsubsection{ Estimation of \eqref{eq:v_2_term} and \eqref{eq:v_1_term}}

We estimate the $\mathbb{L}^2$-norm of the terms \eqref{eq:v_2_term} and \eqref{eq:v_1_term}  using Theorem \ref{thm:vestimate2}  along with the fact that $|\rho_\epsilon(\cdot, s)|\leq 1$.
Note that, from Theorem \ref{thm:vestimate2}, precisely from \eqref{newvestimate} with $n=2,$ we bound the term   \eqref{eq:v_2_term}  from above by
\begin{equation*}
\begin{split}
&\int_{0}^t \frac{j}{2\sqrt \epsilon} H_s(1) v_2^\epsilon(N-1,N,s|\mu^\epsilon)ds\lesssim \frac{\epsilon^{1-2\zeta}tC(H)}{\sqrt \epsilon} .
\end{split}
\end{equation*}
Above and below, $C(H)$ is a constant that depends on $H$.
Similarly, we prove that \eqref{eq:v_1_term} is bounded as 
\begin{equation*}
\int_{0}^t \frac{j}{2\sqrt \epsilon} H_s(1) \Big\{\rho_\epsilon(N,s)v_1^\epsilon(N-1,s|\mu^\epsilon)+\rho_\epsilon(N-1,s)v_1^\epsilon(N,s|\mu^\epsilon)\Big\}ds\lesssim \epsilon^{1/2-2\zeta}tC(H).
\end{equation*}
For $\zeta<1/4$, the $\mathbb L^2$- norm of the above terms vanish as $\eps\to0$.

\subsubsection{ Estimation of \eqref{eq:diff_eta}-\eqref{degree_four}} 
Unlike the terms in \eqref{eq:v_2_term} and \eqref{eq:v_1_term},  which are deterministic, the remaining terms do not involve space correlation functions $v_{n}^{\epsilon}$ but instead, they depend on products of  $\bar\eta$'s.  
Since these terms are more intricate than the previous ones,  we now provide a comprehensive analysis of estimating \eqref{eq:diff_eta}–\eqref{degree_four}.   Lemmas \ref{lem:diff_eta}--\ref{lem:estimation of v2} provide the precise estimates  which are proved by primarily  using Theorem \ref{thm:vestimate3} and  results from \cite{de2011current} highlighted in the subsequent analysis. Each estimate is presented in detail below, demonstrating that the terms  \eqref{eq:diff_eta}–\eqref{degree_four} vanish in the $\mathbb{L}^2$-norm as $\epsilon\to0$. This is provided under a sufficiently small  parameter $\zeta>0$ appearing  in Lemmas \ref{lem:diff_eta}--\ref{lem:estimation of v2}.

The estimate of the  $\mathbb{L}^2$-norm  of \eqref{eq:diff_eta} is given in the next lemma.
\begin{lemma}\label{lem:diff_eta}
Let $t\in(0,T]$. For any $x,y\in I_{+}$ (similarly for $x,y\in I_{-}$) with $|x-y|=1$ and for any $\zeta>0$, it holds 
\begin{eqnarray*}
&&\hspace{-.7cm}\mathbb{E}^{\epsilon}_{\mu^\epsilon} \left[\left(\int_{0}^t \frac{j}{2\sqrt \epsilon} H_s(1) \rho_\epsilon(x,s)[\bar\eta_s(y)-\bar\eta_s(x)]ds\right)^2\right]\lesssim  \eps^{1-4\zeta}\left(t^{1+\zeta }\mathbf 1_{t<1}+t^{2}\mathbf 1_{t\geq 1}\right).
\end{eqnarray*} 
\end{lemma}

{\begin{remark}
We note that the previous lemma also holds when the points $x,y\notin I_{\pm}$, because it is a consequence of the bounds of Corollary \ref{cor:vestimate3_difference} which hold more generally. 
\end{remark}}
\begin{proof} Observe that by squaring the time integral and using Fubini's theorem, we can bound the expectation appearing in the statement as 
\begin{eqnarray}\label{first calculation for difference of etas}
&&\hspace{-2cm}\mathbb{E}^{\epsilon}_{\mu^\epsilon} \left[\left(\int_{0}^t \frac{j}{2\sqrt \epsilon} H_s(1) \rho_\epsilon(x,s)[\bar\eta_s(y)-\bar\eta_s(x)]ds\right)^2\right]\nonumber\\
&&\lesssim\frac{1}{\epsilon}\int_0^t dr\int_0^rds\Big|\mathbb{E}^{\epsilon}_{\mu^\epsilon} \left[\left(\bar\eta_s(x)-\bar\eta_s(y)\right) \left(\bar\eta_r(x)-\bar\eta_r(y)\right)\right]\Big|\nonumber\\
&&=\frac{1}{\epsilon}\int_0^t dr\int_0^rds \left|v^{\epsilon,1,1}_{s,r}(x,x)-v^{\epsilon,1,1}_{s,r}(x,y)-v^{\epsilon,1,1}_{s,r}(y,x)+v^{\epsilon,1,1}_{s,r}(y,y)\right|\nonumber\\
&&\lesssim \frac{1}{\epsilon}\int_0^t dr\int_0^rds\sup_{\substack{x,y\in\Lambda_N\\x+1\in\Lambda_N}}|v^{\epsilon,1,1}_{s,r}(y, x|\mu^{\epsilon})-v^{\epsilon,1,1}_{s,r}(y, x+1|\mu^{\epsilon})|\nonumber\\
&&\leq \frac{1}{ \epsilon}\int_{0}^t dr \int_0^r  ds{\sup_{\substack{x,y\in\Lambda_N\\x+1\in\Lambda_N}}}|v^{\epsilon,1,1}_{s,r}(y, x|\mu^{\epsilon})-v^{\epsilon,1,1}_{s,r}(y, x+1|\mu^{\epsilon})|{\bf 1}_{r-s\leq \tilde\Delta}\nonumber\\
&&+\frac{1}{ \epsilon}\int_{0}^t dr \int_0^r  ds\sup_{\substack{x,y\in\Lambda_N\\x+1\in\Lambda_N}}|v^{\epsilon,1,1}_{s,r}(y, x|\mu^{\epsilon})-v^{\epsilon,1,1}_{s,r}(y, x+1|\mu^{\epsilon})|{\bf 1}_{r-s> \tilde\Delta}\nonumber\\ &&\lesssim\eps^{1-4\zeta} t^{1+\zeta}+\eps^{1-4\zeta} t^{2}\lesssim \eps^{1-4\zeta} \left(t^{1+\zeta }\mathbf 1_{t<1}+t^{2}\mathbf 1_{t\geq 1}\right),
\end{eqnarray}
where $\tilde\Delta$ is given in Theorem \ref{thm:vestimate3}, and the last line follows from Corollary \ref{cor:vestimate3_difference}. 
\end{proof}
Now we head up to estimating the  $\mathbb{L}^2$-norm  of \eqref{diff_of_rho_epsilon} with the use of the next lemma. 
\begin{lemma}\label{lem:diff_of_rho_epsilon}
Let $t\in(0,T]$. For any $\zeta>0$, it holds
\begin{eqnarray}
&&\mathbb{E}^{\epsilon}_{\mu^\epsilon} \left[\left(\int_{0}^t \frac{j}{2\sqrt \epsilon} H_s(1)(\rho_\epsilon(N-1,s)-\rho_\epsilon(N,s))\bar\eta_s(N)ds\right)^2\right]\lesssim \epsilon^{1-4\zeta}t^{2}.
\end{eqnarray} 
\end{lemma}
\begin{proof}
The lemma follows directly from Corollary \ref{corollary: difference of rhos}, i.e., 
\begin{eqnarray}\label{calculation for difference of rhos}
&&\mathbb{E}^{\epsilon}_{\mu^\epsilon} \left[\left(\int_{0}^t \frac{j}{2\sqrt \epsilon} H_s(1)\Big(\rho_\epsilon(N-1,s)-\rho_\epsilon(N,s)\Big)\bar\eta_s(N)ds\right)^2\right]\nonumber\\
&&\hspace{1cm}\lesssim\frac{1}{\epsilon}\mathbb{E}^{\epsilon}_{\mu^\epsilon} \left[\left(\int_0^t \Big(\rho_\epsilon(N-1,s)-\rho_\epsilon(N,s)\Big)\bar\eta_s(N) ds \right)^2\right]\nonumber\\
&&\hspace{1cm}=2\frac{1}{\epsilon}\int_0^t dr\int_0^rds \Big(\rho_\epsilon(N-1,s)-\rho_\epsilon(N,s)\Big)\Big(\rho_\epsilon(N-1,r)-\rho_\epsilon(N,r)\Big)v^{\epsilon,1,1}_{s, r}(N,N)\nonumber\\
%&=&\int_0^t dr\int_0^rds\mathbb{E}^{\epsilon}_{\mu^\epsilon} \left[\left(\bar\eta(x,s)-\bar\eta(y,s)\right) \bar\eta(x,r)\right]\nonumber\\
%&&\hskip1.5cm-\int_0^t dr\int_0^rds\mathbb{E}^{\epsilon}_{\mu^\epsilon} \left[\left(\bar\eta(x,s)-\bar\eta(y,s)\right) \bar\eta(y,r)\right]\nonumber\\
&&\hspace{1cm}\leq c\frac{1}{\epsilon}\int_0^t dr\int_0^rdsC_T^2\epsilon^{2-4\zeta}\lesssim \epsilon^{1-4\zeta}t^{2}
\end{eqnarray}
where we use that  $|v^{\epsilon,1,1}_{s, r}(N,N)|\leq 1$.
\end{proof}
 The estimation of \eqref{degree_one_plus_hydro_limit} is proved in the following lemma.
\begin{lemma}
\label{lem:use_of_hydrodynamic_limit}
Let $t\in(0,T]$. For any $\zeta>0$, it holds
\begin{eqnarray}
&&\lim_{\epsilon\to 0}\mathbb{E}^{\epsilon}_{\mu^\epsilon} \left[\left(\int_{0}^t \frac{j}{2\sqrt \epsilon} H_s(1)\Big(\rho(1,s)-\rho_\epsilon(N,s)\Big)\bar\eta_s(N)ds\right)^2\right]=0.
\end{eqnarray} 
\end{lemma}
\begin{proof}
By Theorem \eqref{thm:vestimate3} with $n=m=1$, we can bound from above the expectation appearing in the statement of the lemma  by a constant times
\[
\int_0^t dr\int_0^r ds\,
\left|\rho(1,s)-\rho^\varepsilon(N,s)\right|
\left|\rho(1,r)-\rho^\varepsilon(N,r)\right|
\left[
(r-s)^{-1/2}
\mathbf{1}_{\{r-s\leq\widetilde{\Delta}\}}
+
\mathbf{1}_{\{r-s>\widetilde{\Delta}\}}
\right].
\]
By the hydrodynamic limit in \eqref{eq:hydro_conv}, for every $s>0$, the fact that  both $\rho$ and $\rho^\varepsilon$ are bounded  and that the function
\[
(s,r)\longmapsto
(r-s)^{-1/2}
\mathbf{1}_{\{r-s\leq\widetilde{\Delta}\}}
+
\mathbf{1}_{\{r-s>\widetilde{\Delta}\}}
\]
is integrable on the set
$\{(s,r):0<s<r<t\},$
it is enough to apply dominated convergence theorem. 
\end{proof}

Finally, we estimate the  $\mathbb{L}^2$-norm  of \eqref{degree_four} which is a consequence of the next lemma.
\begin{lemma}\label{lem:estimation of v2}
Let $t\in(0,T]$. Then, for any $\zeta>0$ there is a constant $c>0$ such that
\begin{eqnarray}
&&\mathbb{E}^{\epsilon}_{\mu^\epsilon} \left[\left(\int_0^t\frac{j}{2\sqrt \epsilon} H_s(1)\bar\eta_s(N)\bar\eta_s(N-1)\,ds\right)^2\right]\lesssim \epsilon^{\zeta}\left(t^{1+\zeta }\mathbf 1_{t<1}+t^{2}\mathbf 1_{t\geq 1}\right).
\end{eqnarray} 
\end{lemma}
\begin{proof}
The estimate is a direct result of Theorem \ref{thm:vestimate3} for $n=m=2$, i.e., 
\begin{equation*}
\begin{split}
\mathbb{E}^{\epsilon}_{\mu^\epsilon} \left[\left(\int_0^t\frac{j}{2\sqrt \epsilon} H_s(1)\bar\eta_s(N)\bar\eta_s(N-1)\, ds\right)^2\right]
&\lesssim\frac{1}{ \epsilon}\int_{0}^t dr \int_0^r  ds |v_{s, r}^{\epsilon,2,2}(N-1, N, N-1, N)|\nonumber\\
&\leq \frac{1}{ \epsilon}\int_{0}^t dr \int_0^r  ds\sup_{\underline{x}\in\Lambda_N^{2\neq}}|v_{s, r}^{\epsilon,2,2}(\underline{x}, N-1, N)|\nonumber\\
&\leq \frac{1}{ \epsilon}\int_{0}^t dr \int_0^r  ds\sup_{\underline{x}\in\Lambda_N^{2\neq}}|v_{s, r}^{\epsilon,2,2}(\underline{x}, N-1, N)|{\bf 1}_{r-s\leq \tilde\Delta}\nonumber\\
&+\frac{1}{ \epsilon}\int_{0}^t dr \int_0^r  ds\sup_{\underline{x}\in\Lambda_N^{2\neq}}|v_{s, r}^{\epsilon,2,2}(\underline{x}, N-1, N)|{\bf 1}_{r-s> \tilde\Delta}\nonumber\\
&\lesssim \epsilon^{\zeta}t^{1+\zeta}
+\epsilon^{\zeta}t^2
\lesssim \epsilon^{\zeta}
\left(t^{1+\zeta}\mathbf{1}_{t<1}
+t^2\mathbf{1}_{t\geq 1}\right).
\end{split}\end{equation*}
\end{proof}
We now complete the closure of the martingale decomposition for $K=2$ given in \eqref{intpartofmart_1} i.e.
\begin{align}
M_t^\epsilon(H)& \;:=\;\ Y_t^\epsilon (H_t)- Y^\epsilon _0(H_0)-\int_{0}^t  Y_s^\epsilon\Big(\frac 12\partial_u^2 H_s+\partial_s H_s\Big)ds.\end{align}
We observe that the last display can also be generalized to any $K$ using the results of Appendix  \ref{ap:action_gen}, together with the results of  Theorem \ref{thm:vestimate2} and \ref{thm:vestimate3}.
Now we discuss the conditions that we have to impose on the function $H$ in order to recover \eqref{characterization}. 
Observe that for $H$ a $C^2$ function, if we take the function $T_{t,s}H$ where $T_{t,s}$ solves \eqref{eq:Robin_equation_line}, then  from the first equality in  \eqref{eq:Robin_equation_line}, the rightmost term on the first line of last display vanishes, since $\partial_s T_{t,s}H+\frac{1}{2}\partial_u^2T_{t,s}H=0$. Above, we have already used the boundary conditions  in \eqref{eq:Robin_equation_line} to control the contribution from the boundary generator.  Up to here we have that 
\begin{align}
M_t^\epsilon(H)=\ Y_t^\epsilon (H)- Y^\epsilon _0(T_{t,0}H)
\end{align}
 plus terms that vanish as $\epsilon\to 0$. 
 Applying the same martingale central limit argument used in the proof of Corollary \ref{cor:conv_mart} to the time-dependent test function $H_s=T_{t,s}H$, we obtain
the convergence of 
$M_t^\epsilon(T_{t,\cdot}H)$ to a  mean-zero Gaussian random variable $W_t(H)$ which has  variance
\[
\mathbb E[W_t(H)^2]
=
\int_0^t\|T_{t,s}H\|_{L^2(\rho_s)}^2\,ds.
\]
where $\rho(u,s)$  is the solution of the hydrodynamic equation \eqref{eq:Robin_equation} and the $\mathbb L^2$-norm is defined in \eqref{eq:norm}. Moreover, if we assume also that the sequence $\{Y_t^\epsilon\}_\epsilon$  is tight, which will be proved in the next section, then taking the limit through a subsequence $\{Y_t^\epsilon\}_{\epsilon_k}$, we see that the limit points satisfy \eqref{characterization} 
  as desired.

\section{Tightness}\label{sec:tightness}

In order to prove tightness, we recall that
$\mathbb S=C^\infty([-1,1])$ is a nuclear Fr\'echet space when endowed
with the seminorms defined in \eqref{semi-norm}. As a consequence, by
Mitoma's criterion (that we recall below), it is enough to show tightness of the sequence of
real-valued processes $\{Y_t^\epsilon(H)\}_{\epsilon>0}$, for every
$H\in\mathbb S$.

\begin{theorem}[Mitoma's criterion - Theorem 4.1 of \cite{mitoma1983tightness}]
A sequence of processes $\{X^\epsilon_t; t \in [0,T]\}_{\epsilon>0}$ in $\mathcal{D}([0,T], \mathbb {S}')$ is tight with respect to the Skorohod topology if, and only if, for every $H \in \mathbb{S}$, the sequence of real-valued processes $\{X^\epsilon_t(H); t \in [0,T]\}_{\epsilon}$ is tight with respect to the Skorohod topology of $\mathcal {D}([0,T], \mathbb{R})$.
\end{theorem}
Recall \eqref{martingaleM}. In order to show that $\{Y^\epsilon_t(H)\}_{\epsilon>0}$ is tight, it is enough to show that  \begin{equation*}
\{Y^\epsilon_0(H)\}_{\epsilon>0}
 	\ , \ \{[M^\epsilon _t(H)]_{t \geq 0}\}_{\epsilon>0} 
		 \textrm{ and } \left\{\int_0^t  L_\epsilon Y^\epsilon_s(H) ds \right\}_{\epsilon>0}
\end{equation*} are tight. 
\subsection{Initial time}
In this subsection  we show that $\{Y^\epsilon_0(H)\}_{\epsilon>0}$ is tight.
To this end, it is enough to show that
\begin{equation*}
    \lim_{A \to \infty} \limsup_{\epsilon\to 0} \mathbb{P}_{\mu^\epsilon} [|Y_0^\epsilon(H)| > A] = 0.
\end{equation*}

From \eqref{density field}  and  from
 Markov's inequality, for every $A > 0$ and for every $\epsilon>0$, 
\begin{align*}
\mathbb{P}_{\mu^\epsilon} [|Y_0^\epsilon(H)| > A] &\leq \frac{1}{A^2} \mathbb{E}^\eps_{\mu^\epsilon}[|Y_0^\epsilon(H)|^2] \\
&= \frac{\epsilon}{A^2}  \Big( \sum_{x\in \Lambda_N} [H(\epsilon x) ]^2 \mathbb{E}^\eps_{\mu^\epsilon} [(\eta_{0}(x)-\rho^\epsilon_0(x))^2] + \sum_{\substack{x, y\in \Lambda_N\\ y \neq x}} H(\epsilon x)H(\epsilon y)C^{2,\epsilon}_{0}(x,y) ,
\end{align*}
where $C_t^{2,\epsilon}(x,y)$ was defined in \eqref{eq:correlation}. Note that $C^{2,\epsilon}_{0}(x,y)\equiv0$ given the choice of our initial measure, see \textbf{Assumption 1}. Therefore, the rightmost term in last display is equal to zero, and  
the leftmost term in last display is of order $O(A^{-2})$ and so the limit is proved.  
In fact, we observe that we can prove the convergence of the fluctuation field at time $t=0$ to a Gaussian process by using characteristic functions, in the following  manner:
\begin{equation}\begin{split}
\log  \mathbb E^\eps_{\mu^\eps}[\exp\{iY_0^\eps(H)\}]&=\log \prod_{x\in\Lambda_N} \mathbb E^\eps_{\mu^\eps}[\exp\{i\sqrt \eps H(\eps x )(\eta_{0}(x)-\rho^\epsilon_0(x) )\}]  \\&=\sum_{x\in\Lambda_N}\log  \mathbb E^\eps_{\mu^\eps}[\exp\{i\sqrt \eps H(\eps x )(\eta_{0}(x)-\rho^\epsilon_0(x) )\}]  
\end{split}\end{equation}
and by Taylor expansion the last display can be written as :
\begin{equation}
\sum_{x\in\Lambda_N}\log \Big( 1- \frac{\eps}{2} H^2(\eps x )E^\eps_{\mu^\eps}[(\eta_{0}(x)-\rho^\epsilon_0(x) )^2]+O(\eps^{3/2})\Big),  
\end{equation}
but again from Taylor expansion, from {\bf{Assumption 1}} and since $u_0$ is a smooth function the last display converges, as $\eps\to0$, to {$-\frac{1}{2}\int_{-1}^1H^2(u)\chi(u_0(u))du.$} This shows that the characteristic function of the initial density field $Y_0(H)$ coincides with the one of the Gaussian process with the variance given above. 
\subsection{The integral terms}

To prove tightness of the integral terms we use the Kolmogorov-Centsov's criterion, that we state here for convenience.
\begin{proposition}[Kolmogorov-Centsov criterion - Problem 2.4.11 of \cite{karatzas2014brownian}]
\label{KolCen}

A sequence $\{X^\epsilon_t ; t \in [0, T]\}_{ \epsilon>0}$ of
continuous, real-valued, stochastic processes is tight with respect to the uniform topology of $\mathcal C([0, T];\mathbb{R})$ if the sequence of real-valued random variables $\{X^\epsilon_0\}_{\epsilon>0}$ is tight and there are constants $K, \gamma_1,\gamma_2 > 0$ such that, for any $t,s \in [0, T]$ and any $\epsilon>0$, it holds 
\begin{equation*}
    \mathbb{E}^\eps_{\mu^\epsilon}[|X^\epsilon_t - X^\epsilon_s|^{\gamma_1}] \leq K |t-s|^{1+\gamma_2}.
\end{equation*}
\end{proposition}

\subsubsection{The bulk term}
We first treat the bulk term, i.e., the term of the form
$$\mathbf{Z}_t^\epsilon(H):=\int_0^t Y_s^\epsilon (\Delta_\epsilon H)ds.$$
 We note that from \eqref{density field} for any $H\in\mathbb S$, we have that 
\begin{equation*}
Y^\epsilon _s(H)\;=\;\sqrt \epsilon\sum_{x\in\Lambda_N}H(\epsilon x)
\Big(\eta_{s}(x)-\rho_\epsilon(x,s)\Big)+\sqrt \epsilon\sum_{x\in\Lambda_N}H(\epsilon x)
\Big(\rho_\epsilon(x,s)-\rho^\epsilon_s(x)\Big),
\end{equation*}
where $\rho_\eps(s,x)$ is the solution to \eqref{eq:linearized}. 

From \eqref{eq:important_1}, the rightmost term in last display equals to 
\begin{equation*}
-\sqrt \epsilon\sum_{x\in\Lambda_N}H(\epsilon x) v_{{1}}^{\epsilon}(x,s|\mu^{\epsilon})
\end{equation*}
and now we treat separately the two cases.

Since $H\in\mathbb S$, the discrete
Laplacian satisfies
\begin{equation}\label{eq:discrete_laplacian_l1}
\sum_{x\in\Lambda_N}
|\Delta_\epsilon H(\epsilon x)|
\leq \frac{C_H}{\epsilon}.
\end{equation}
Indeed, at interior sites, Taylor expansion gives
$|\Delta_\epsilon H(\epsilon x)|\leq C_H$,
while at the two boundary sites
$|\Delta_\epsilon H(\epsilon x)|\leq C_H/\epsilon$.
Thus, \eqref{eq:discrete_laplacian_l1} holds without
imposing any boundary conditions on $H$. First note that if
\begin{equation}
\mathbf{X}_t^\epsilon=\int_{0}^t ds\, \sqrt \epsilon\sum_{x\in\Lambda_N}\Delta_\epsilon H(\epsilon x)
\Big(\eta_{s}(x)-\rho_\epsilon(x,s)\Big),
\end{equation}
then by squaring the integral we have that 
\begin{equation}
\mathbb{E}^{\epsilon}_{\mu^\epsilon}\left[|\mathbf{X}_t^\epsilon-\mathbf{X}_s^\epsilon|^2\right]=2\epsilon\int_{s}^t\int_{s}^rdr\,d\tau\sum_{x,y\in \Lambda_N}\Delta_{\epsilon}H(\epsilon x)\Delta_{\epsilon}H( \epsilon y){v^{\epsilon,1,1}_{\tau,r}(y, x)}.
\end{equation}

By \eqref{eq:discrete_laplacian_l1} and
Theorem \ref{thm:vestimate3}, the last display
is bounded from above by: 
\begin{eqnarray}\label{tightness_est}
&&\epsilon^{-1}\int_{s}^t\int_{s}^rdr\,d\tau \sup_{x,y} |v^{\epsilon,1,1}_{\tau,r}( x,y)|\\ &\leq&\epsilon^{-1}\int_{s}^t\int_{s}^rdr\,d\tau \sup_{x,y} |v^{\epsilon,1,1}_{\tau,r}( x,y)|\left({\bf 1}_{r-\tau\leq \tilde\Delta}+{\bf 1}_{r-\tau> \tilde\Delta}\right)\nonumber\\
& \lesssim&\left((t-s)^{3/2}+(t-s)^2\right).
\end{eqnarray}
Now, for 
\begin{equation*}
\mathbf{Y}^\epsilon_t=\int_0^t ds\sqrt \epsilon\sum_{x\in\Lambda_N}\Delta_\epsilon H(\epsilon x) v_{1}^{\epsilon}(x,s|\mu^{\epsilon})
\end{equation*}
we evaluate as
\begin{equation}
\mathbb{E}^{\epsilon}_{\mu^\epsilon}\left[|\mathbf{Y}_t^\epsilon-\mathbf{Y}_s^\epsilon|^2\right]=2\epsilon\int_{s}^t\int_{s}^rdr\,d\tau\sum_{x,y\in \Lambda_N}\Delta_{\epsilon}H(\epsilon x)\Delta_{\epsilon}H\,(\epsilon y) v_{1}^{\epsilon}(x,r|\mu^{\epsilon})v_{1}^{\epsilon}(y,\tau|\mu^{\epsilon})
\end{equation}
and as above, by using Theorem \ref{thm:vestimate2}, this last display is bounded from above by 
\begin{equation}
\epsilon^{-1}\int_{s}^t\int_{s}^rdr\,d\tau \sup_{x,y} |v_1^\epsilon(x,r|\mu^\epsilon)| |v_1^\epsilon(y,\tau|\mu^\epsilon) |\lesssim\epsilon^{1-4\zeta}(t-s)^{2}.
\end{equation}
\subsubsection{The boundary terms}

For the boundary terms we define
\begin{equation*}
\widetilde{\mathbf{X}}_t^\epsilon=\int_{0}^t ds\epsilon^{-1/2}\left(\bar\eta_{s}(x)-\bar\eta_{s}(y)\right)
\quad \textrm{and}\quad 
\widehat{\mathbf{X}}_t^\epsilon=\int_{0}^t ds\epsilon^{-1/2}\prod_{x\in\mathfrak X}\bar\eta_{s}(x)
\end{equation*}
where $x,y\in I_\pm$ with $|x-y|=1$,
and $\mathfrak X\subset I_\pm$.
Similar to what we did to treat the bulk term and \eqref{tightness_est}
\begin{eqnarray}
\mathbb{E}^{\epsilon}_{\mu^\epsilon}\left[|\widetilde{\mathbf{X}}_t^\epsilon-\widetilde{\mathbf{X}}_s^\epsilon|^2\right]&=&\epsilon^{-1}\int_{s}^t\int_{s}^rdr\,d\tau\mathbb{E}^{\epsilon}_{\mu^\epsilon}\left[\left(\bar\eta_{r}(x)-\bar\eta_{r}(y)\right)\left(\bar\eta_{\tau}(x)-\bar\eta_{\tau}(y)\right)\right]\nonumber\\
&\lesssim&  \epsilon^{-1}\int_{s}^t\int_{s}^rdr\,d\tau \sup_{w\in\{x,y\}}\left(|v^{\epsilon,1,1}_{\tau,r}(y,w)|+|v^{\epsilon,1,1}_{\tau,r}(x,w)|\right)
\nonumber\\
&\lesssim&\left((t-s)^{3/2}+(t-s)^2\right).
\end{eqnarray}
{The term
$\mathbb{E}^{\epsilon}_{\mu^\epsilon}
\left[
|\widehat{\mathbf{X}}_t^\epsilon-\widehat{\mathbf{X}}_s^\epsilon|^2
\right]$
is estimated analogously using Theorem~\ref{thm:vestimate3}, with
$n=m=1$ when $|\mathfrak X|=1$ and the case $n+m\geq 4$ when
$|\mathfrak X|\geq 2$.}

Finally, we consider the contribution corresponding to
\eqref{degree_one_plus_hydro_limit}. Up to a bounded
multiplicative coefficient, this term is of the form
\[
R_t^\epsilon
=
\frac{1}{\sqrt{\epsilon}}
\int_0^t
\big(\rho(1,r)-\rho_\epsilon(N,r)\big)
\bar\eta_r(N)\,dr.
\]
By Theorem~\ref{thm:vestimate3} with $n=m=1$,
and using the boundedness of $\rho$ and $\rho_\epsilon$,
we obtain, uniformly in $\epsilon$,
\[
\begin{aligned}
\mathbb E_{\mu^\epsilon}^\epsilon
\left[|R_t^\epsilon-R_s^\epsilon|^2\right]
&\lesssim
\int_s^t dr\int_s^r d\tau\,
\left[
(r-\tau)^{-1/2}
\mathbf 1_{\{r-\tau\leq\widetilde\Delta\}}
+
\mathbf 1_{\{r-\tau>\widetilde\Delta\}}
\right]\\
&\lesssim
(t-s)^{3/2}+(t-s)^2.
\end{aligned}
\]
Since $R_0^\epsilon=0$, the preceding estimate
and the Kolmogorov--Centsov criterion imply
that $\{R^\epsilon\}_\epsilon$ is tight.
The corresponding term at the left boundary
is treated analogously.

This concludes the proof of tightness for all the integral terms.
In the next subsection, we establish the tightness of the sequence of martingales. More precisely, we show that the sequence converges to a mean-zero Gaussian process, which in particular implies its tightness.

\subsection{Convergence of the sequence of martingales}
In this section we prove the convergence of the sequence of martingales $\left\{ M_t^\epsilon(H); t\in [0,T]\right\}_{\epsilon>0}$ using  Theorem VIII.3.12  in \cite{jacod2013limit}. In fact, we use the statement of that theorem as in Proposition 4 of \cite{ahmed2022microscopic}:   

\begin{proposition}[\cite{ahmed2022microscopic}]\label{jacod_shir}
Let $t \in [0,T] \mapsto C_t  \in [0, \infty)$ be a deterministic continuous  function of the time $t$. Let $\{M_t^\epsilon ;  t \in [0,T]\}_{\epsilon>0}$ be a sequence of square-integrable real-valued martingales with c\`adl\`ag trajectories defined on a probability space $(\Omega,\mathcal  F, \mathbb P)$. Let $\{\langle M^\epsilon\rangle_t;  t \in [0,T]\}$ denote the quadratic variation of $\{M_t^\epsilon\; ; \; t \in [0,T]\}$. Assume that
\begin{itemize}
\item[i)] For each $\epsilon>0$, the quadratic variation process $\{\langle M^\epsilon\rangle_t;  t \in [0,T]\}$ has continuous trajectories $\mathbb P$ a.s.;
\item[ii)] the maximal jump satisfies
\begin{equation}\label{eq:maximal_jump}
\lim_{\epsilon\to 0} \mathbb E\Big[ \sup_{0 \le s \le T} \big|M_s^\epsilon-M_{s-}^\epsilon\big| \Big] =0;
\end{equation}
Above $\mathbb E$ denotes the expectation w.r.t. $\mathbb P$.
\item[iii)] For each $t \in [0,T]$, the sequence of random variables $\{\langle M^\epsilon
\rangle_t \}_{\epsilon>0}$ converges in probability to the deterministic path  $\{C_t \; ;\;  t \in [0,T]\}$.
\end{itemize}
Then the sequence $\{M_t^\epsilon;  t \in [0,T]\}_{\epsilon>0}$ converges in law in $\mathcal D ([0,T], \mathbb R)$ to a martingale $\{M_t; t \in [0,T]\}$ with quadratic variation $t \mapsto C_t$. Moreover $\{M_t; t \in [0,T]\}$ is a mean zero Gaussian process. 
\end{proposition}
From the previous result we conclude that:
\begin{corollary}\label{cor:conv_mart}
For $H\in\mathbb S$ the sequence of martingales
${M_t^\epsilon(H);t\in [0,T]}{\epsilon}$ converges in the topology of
$\mathcal D([0,T],\mathbb R)$, as $\epsilon\to 0$, towards a mean-zero Gaussian process
$W_t(H)$ with quadratic variation given by
\begin{equation*}
\int_0^t
\|H\|_{L^2(\rho_s)}^2\,ds,
\end{equation*}
where $\rho(u,s)$ is the solution of the hydrodynamic equation
\eqref{eq:Robin_equation} and the $\mathbb L^2$-norm is defined in
\eqref{eq:norm}.
\end{corollary}

\begin{proof}
The proof of Proposition \ref{jacod_shir} is a simple consequence of  \cite[Theorem VIII.3.12]{jacod2013limit}. Now we verify the three conditions of Proposition \ref{jacod_shir}.
First,  we observe that ii) is a consequence of the fact that  
\begin{equation*}
\lim_{\epsilon\to 0} \mathbb E^\epsilon_{\mu^\eps}\Big[ \sup_{s\le t} \big| M^\epsilon_s(H) -
M^\epsilon_{s^-}(H)\big| \Big]\;=\; 0,
\end{equation*}
since at each time of a ring of a clock of a Poisson process, the configuration $\eta$ only changes its value in at most two sites, which in turn implies that 
\begin{equation*}
\lim_{\epsilon\to 0}\sup_{s\le t} \big| M_s^\epsilon(H) -
M^\epsilon_{s^-}(H)\big| =\lim_{\epsilon\to 0}\sup_{s\le t} \big| Y_s^\epsilon(H) -
Y^\epsilon_{s^-}(H)\big| \leq \lim_{\epsilon\to 0} \sqrt \epsilon 2 \|H\|_\infty=0.
\end{equation*}
For items i) and iii), we note that from \eqref{quadratic}, the quadratic variation of the martingale $M_s^\epsilon(H)$ is given by 
$\int_0^t \Gamma^\epsilon_s(H)ds$ and 
a simple computation shows  that 
\begin{equation*}
\begin{split}
\Gamma_s^\epsilon(H)=\frac{\epsilon}{2} \sum_{x=-N}^{N-1} (\nabla^+_\epsilon H(\epsilon x))^2 (\eta_{s}(x)-\eta_{s}(x+1))^2+\frac{j}{2}\sum_{x\in I_\pm}D_\pm\eta_{s}(x) (H(\epsilon x))^2.
\end{split}
\end{equation*}
Item i) follows from the fact that the number of particles is bounded  and from the observation that the integral in time of a bounded function is a continuous function  $\mathbb P$ a.s.   
To analyse the convergence, we observe that, 
\begin{equation}\label{eq:boundary_terms_new}
 \sum_{x\in I_+ } D_+  \eta (x)=1{-}\prod_{x\in I_+}\eta (x)
\end{equation}
and 
\begin{equation}\label{eq:boundary_terms_new_1}
 \sum_{x\in I_- } D_-  \eta (x)=\sum_{j=1}^K(-1)^{j+1}\sum_{A_j\subset I_-} \prod_{x\in A_j}\eta(x),
\end{equation}
where for $j=1,\dots, K$ the sets $A_j\subset I_-$ with cardinality $j
$. Since the measure $\mu^\epsilon $ satisfies \eqref{eq:mea_associated}, then, from the Replacement Lemma proved in  \cite[Lemma A.3]{erignoux2020hydrodynamics} and by the hydrodynamic limit proved in  \cite[Theorem 2.7]{erignoux2020hydrodynamics}  we conclude that 
$
\int_{0}^t\Gamma_s^\epsilon( H)ds,
$
converges in distribution, as $\epsilon\to 0$, to 
\begin{equation*}\label{eq:covariance}
\begin{split}
 & \int_0^t \int_{-1}^1\chi(\rho_s(u)) \big(\partial_u H(u)\big)^2du \; ds \\
  +&\int_0^t \left\{\frac{j}{2}(1-\rho_s(1)^{K})(H(1))^2+ \frac{j}{2}\Big[1-(1-\rho_s(-1))^{K}\Big](H(-1))^2\right\}\,ds =\int_0^t\|H\|_{L^2(\rho_s)}^2ds.
\end{split}
\end{equation*}
Since the last expression is deterministic, the convergence holds in probability. Hence condition iii) of Proposition~\ref{jacod_shir} is satisfied. Therefore, conditions i)--iii) of Proposition~\ref{jacod_shir} hold, which concludes the proof.
\end{proof}

\appendix
\section{Appendix}
In this appendix we collect several computations that have been used throughout the article. We begin by presenting explicit calculations concerning the action of the generator on the occupation variables.
\subsection{The action of the generator}\label{ap:action_gen}
A simple computation shows that for $x=-(N-1),\dots, N-1$$$L_0\eta(x)= \frac 12\Big(\eta(x+1)+\eta(x-1)-2\eta(x)\Big)$$ and 
$L_0\eta(N)= \frac 12\Big(\eta(N-1)-\eta(N)\Big)$ and $L_0\eta(-N)=\frac 12\Big(\eta(-(N-1))-\eta(-N)\Big)$. To be consistent with the notation introduced above, see \eqref{def_laplacian} we write $L_0\eta(x)=\frac 12\Delta \eta(x). $
Now, for  the boundary action, we note that a simple computation shows that
$L_{b,\pm}\eta(x)=\frac{j}{2}D_\pm\eta(x)(1-2\eta(x))$
for $x\in I_\pm$.
\smallskip

{\textbf{The discrete profile:}}
Recall \eqref{eq:exp_eta}.
From the previous computations and Kolmogorov's equation we  can conclude that 
\begin{equation}\label{eq:kolm_disc_prof}
\partial_t \rho_t^\epsilon (x)=\frac 12\Delta_\epsilon  \rho_t^\epsilon (x)+\frac {\epsilon^{-1}j}{2} \textbf{1}_{x\in I_{+}}\mathbb E^\epsilon_{\mu^\eps}[D_+\eta_{t}(x)]- \frac {\epsilon^{-1}j}{2} \textbf{1}_{x\in I_{-}}\mathbb E^\epsilon_{\mu^\eps}[D_-\eta_{t}(x)].
\end{equation}

{\textbf{The action on the density field:}}
Now we compute the action of the generator on the density fluctuation field. A simple computation shows that 
$$L_0Y^\epsilon _s(H) =\frac{\sqrt\epsilon}{2}  \sum_{x\in \Lambda_N} H(\epsilon x) \Delta \eta_{s} (x)=\frac{\sqrt \epsilon}{2} \sum_{x\in \Lambda_N} \Delta H(\epsilon x) \eta_{s} (x),
$$
$$L_{b,\pm}Y^\epsilon _s(H) =\pm\sqrt \epsilon \frac j2 \sum_{x\in I_\pm } H(\epsilon x) D_\pm  \eta_{s} (x).
$$
Recall \eqref{eq:gen_total}, then
$$L_\epsilon Y^\epsilon _s(H) =\frac{\sqrt \epsilon}{2} \sum_{x\in \Lambda_N} 
\Delta_\epsilon H(\epsilon x) \eta_{s} (x)+\frac {j}{2\sqrt \epsilon} \sum_{x\in I_\pm } H(\epsilon x) D_\pm  \eta_{s} (x).$$
Moreover, from \eqref{eq:kolm_disc_prof} we see that

\begin{equation*}
\begin{split}\partial_s Y^\epsilon _s(H) &=-\frac{\sqrt{\epsilon}}{2}\sum_{x\in \Lambda_N} H(\epsilon x) \Delta_\epsilon  \rho_s^\epsilon (x)\\&-\frac{j}{2\sqrt \epsilon}\sum_{x\in I_+} H(\epsilon x)\mathbb E^\epsilon_{\mu^\eps}[D_+\eta_{s}(x)]+ \frac{j}{2\sqrt \epsilon}\sum_{x\in I_-} H(\epsilon x)\mathbb E^\epsilon_{\mu^\eps}[D_-\eta_{s}(x)].\end{split}\end{equation*}

By a summation by parts last expression becomes 
\begin{equation*}
\begin{split}\partial_s Y^\epsilon _s(H) &=-\frac{\sqrt{\epsilon}}{2}\sum_{x\in \Lambda_N} \Delta_\epsilon H(\epsilon x)  \rho_s^\epsilon (x)\\&-\frac{j}{2\sqrt \epsilon}\sum_{x\in I_+} H(\epsilon x)\mathbb E^\epsilon_{\mu^\eps}[D_+\eta_{s}(x)]+ \frac{j}{2\sqrt \epsilon}\sum_{x\in I_-} H(\epsilon x)\mathbb E^\epsilon_{\mu^\eps}[D_-\eta_{s}(x)]. \end{split}\end{equation*}

\subsection{Treating boundary terms}\label{ap:boundary_terms}
Recall that, just below \eqref{eq:boundary_term_final}, and in order to give a more transparent presentation for the closure of the martingale, we presented it for the case $K=2$. Below, and for completeness, we present the computations for general $K$. To that end,  recall  the form of $D_+$ given in \eqref{eq:boundary_operators}. For any $s\in[0,T]$, we have the identity
\begin{equation}\label{eq:boundary_terms}
 \sum_{x\in I_+ } \Big(D_+  \eta_s (x)-\mathbb E^\epsilon_{\mu^\eps}[D_+\eta_s(x)]\Big)=-\Big(\prod_{x\in I_+}\eta_s (x)-\mathbb{E}^\eps_{\mu^\epsilon}\Big[\prod_{x\in I_+}\eta_s (x)\Big]\Big).
\end{equation}

For $j=1,\dots, K$, denoting by $A_j$ the subsets of $I_+$ with cardinality equal to $j$, we get 
\begin{equation*}\label{eq:prod_eta}\begin{split}
 \prod_{x\in I_+}\eta (x)&=\sum_{j=1}^K\sum_{A_j\subset I_+} \Big\{\prod_{x\in A_j} \bar\eta_s(x)\prod_{y\in A_j^c}\rho_\epsilon(y,s)\Big\}+\prod_{x\in I_+}\rho_\epsilon(x,s)\\&=\sum_{x\in I_+}  \Big\{\bar\eta_s(x)\prod_{y\in I_+\setminus \{x\}}\rho_\epsilon(y,s)\Big\}+\sum_{j=2}^K\sum_{A_j\subset I_+} \Big\{\prod_{x\in A_j} \bar\eta_s(x)\prod_{y\in A_j^c}\rho_\epsilon(y,s)\Big\}+\prod_{x\in I_+}\rho_\epsilon(x,s).\end{split}
\end{equation*}
Now we focus on the  leftmost term on the last line of last display and we rewrite it as 
\begin{equation*}\begin{split}
&\sum_{x\in I_+\setminus\{N\}}  \Big\{\bar\eta_s(x)\prod_{y\in I_+\setminus \{x\}}\rho_\epsilon(y,s)\Big\}+\bar\eta_s(N)\prod_{y\in I_+\setminus\{N\}}\rho_\epsilon(y)\\=&\sum_{x\in I_+\setminus\{N\}}  \Big\{(\bar\eta_s(x)-\bar\eta_s(N))\prod_{y\in I_+\setminus \{x\}}\rho_\epsilon(y,s)\Big\}+\bar\eta_s(N)\sum_{x\in I_+}  \prod_{y\in I_+\setminus \{x\}}\rho_\epsilon(y,s).\end{split}
\end{equation*}
We write the rightmost term of the last display as
\begin{equation*}
\begin{split}
&\bar\eta_s(N)\sum_{x\in I_+}
\prod_{y\in I_+\setminus\{x\}}\rho_\epsilon(y,s)\\
&=K\bar\eta_s(N)(\rho_\epsilon(N,s))^{K-1}
+\bar\eta_s(N)R_\epsilon(s),
\end{split}
\end{equation*}
where
\begin{equation}\label{eq:R_epsilon}
R_\epsilon(s):=
\sum_{x\in I_+}
\left[
\prod_{y\in I_+\setminus\{x\}}\rho_\epsilon(y,s)
-(\rho_\epsilon(N,s))^{K-1}
\right].
\end{equation}
Since $0\leq\rho_\epsilon(y,s)\leq 1$, the
telescopic product identity gives
\begin{equation*}
\begin{split}
|R_\epsilon(s)|
&\leq
\sum_{x\in I_+}
\sum_{y\in I_+\setminus\{x\}}
|\rho_\epsilon(y,s)-\rho_\epsilon(N,s)|\\
&=(K-1)\sum_{y\in I_+}
|\rho_\epsilon(y,s)-\rho_\epsilon(N,s)|.
\end{split}
\end{equation*}
By Corollary~\ref{corollary: difference of rhos},
for fixed $K$ and $0\leq s\leq T$,
\begin{equation}\label{eq:R_epsilon_bound}
\sup_{0\leq s\leq T}|R_\epsilon(s)|
\leq C_{T,K}\epsilon^{1-2\zeta}.
\end{equation}

Putting together all the previous computations we get 
\begin{equation*}\begin{split}
 \prod_{x\in I_+}\eta_s (x)&=\sum_{j=2}^K\sum_{A_j\subset I_+} \Big\{\prod_{x\in A_j} \bar\eta(x)\prod_{y\in A_j^c}\rho_\epsilon(y,s)\Big\}+\prod_{x\in I_+}\rho_\epsilon(x,s)\\&+\sum_{x\in I_+\setminus\{N\}}  \Big\{(\bar\eta_s(x)-\bar\eta_s(N))\prod_{y\in I_+\setminus \{x\}}\rho_\epsilon(y,s)\Big\}\\&+\bar\eta_s(N)R_\epsilon(s)+K\bar\eta_s(N)(\rho_\epsilon(N,s))^{K-1}.\end{split}
\end{equation*}
From this we conclude that  
\begin{equation}\label{eq:prod_eta_final}\begin{split}
 \prod_{x\in I_+}\eta_s(x)&=\sum_{j=2}^K\sum_{A_j\subset I_+} \Big\{\prod_{x\in A_j} \bar\eta_s(x)\prod_{y\in A_j^c}\rho_\epsilon(y,s)\Big\}+\prod_{x\in I_+}\rho_\epsilon(x,s)\\&+\sum_{x\in I_+\setminus\{N\}}  \Big\{(\bar\eta_s(x)-\bar\eta_s(N))\prod_{y\in I_+\setminus \{x\}}\rho_\epsilon(y,s)\Big\}\\&+\bar\eta(N)R_\epsilon(s)+K\bar\eta_s(N)(\rho_\epsilon(N,s))^{K-1}.\end{split}
\end{equation}
Observe that the expectation of the last display is equal to 
\begin{equation}\label{eq:exp_prod_eta_final}\begin{split}
\mathbb E^\epsilon_{\mu^\eps}\Big[ \prod_{x\in I_+}\eta_s (x)\Big]&=\sum_{j=2}^K\sum_{A_j\subset I_+} \Big\{\mathbb  E^\epsilon_{\mu^\eps}\Big[\prod_{x\in A_j} \bar\eta_s(x)\Big]\prod_{y\in A_j^c}\rho_\epsilon(y,s)\Big\}+\prod_{x\in I_+}\rho_\epsilon(x,s)
\\&+\sum_{x\in I_+\setminus\{N\}}  \Big\{(\mathbb E^\epsilon_{\mu^\eps}[\bar\eta_s(x)]-\mathbb E^\epsilon_{\mu^\eps}[\bar\eta_s(N)])\prod_{y\in I_+\setminus \{x\}}\rho_\epsilon(y,s)\Big\}\\&+\mathbb E^\epsilon_{\mu^\eps} [\bar\eta_s(N)]R_\epsilon(s)+K\mathbb E^\epsilon_{\mu^\eps}[\bar\eta_s(N)](\rho_\epsilon(N,s))^{K-1}.
\end{split}
\end{equation}
Putting together \eqref{eq:prod_eta_final} and \eqref{eq:exp_prod_eta_final} we see that 
\begin{equation}\label{eq:boundary_terms_final}\begin{split}
 &\sum_{x\in I_+ } \Big(D_+  \eta_s (x)-\mathbb E^\epsilon_{\mu^\eps}[D_+\eta_s(x)]\Big)\\
=-&\sum_{j=2}^K\sum_{A_j\subset I_+} \Big\{\prod_{x\in A_j} \bar\eta_s(x)\prod_{y\in A_j^c}\rho_\epsilon(y,s)\Big\}+\sum_{j=2}^K\sum_{A_j\subset I_+} \Big\{\mathbb E^\epsilon_{\mu^\eps}\Big[\prod_{x\in A_j} \bar\eta_s(x)\Big]\prod_{y\in A_j^c}\rho_\epsilon(y,s)\Big\}
\\-&\sum_{x\in I_+\setminus\{N\}}  \Big\{(\bar\eta_s(x)-\bar\eta_s(N))\prod_{y\in I_+\setminus \{x\}}\rho_\epsilon(y,s)\Big\}\\+&\sum_{x\in I_+\setminus\{N\}}  \Big\{(\mathbb E^\epsilon_{\mu^\eps}[\bar\eta_s(x)]-\mathbb E^\epsilon_{\mu^\eps}[\bar\eta_s(N)])\prod_{y\in I_+\setminus \{x\}}\rho_\epsilon(y,s)\Big\}
\\-&\bar\eta_s(N)R_\epsilon(s)-K\bar\eta_s(N)(\rho_\epsilon(N,s))^{K-1}\\+&\mathbb E^\epsilon_{\mu^\eps} [\bar\eta_s(N)]R_\epsilon(s)+K\mathbb E^\epsilon_{\mu^\eps}[\bar\eta(N,s)](\rho_\epsilon(N,s))^{K-1}.
\end{split}
\end{equation}
For the left boundary the computations are completely analogous and for that reason they are omitted.
Now we are back to \eqref{eq:boundary_term_final}. 
From the above computations,  the leftmost term in \eqref{eq:boundary_term_final} can be rewritten as  
\begin{equation*}
\begin{split}
&\underbrace{-\frac{j}{2\sqrt \epsilon} H_s(1)\sum_{j=2}^K\sum_{A_j\subset I_+} \Big\{\prod_{x\in A_j} \bar\eta_{s}(x)\prod_{y\in A_j^c}\rho_\epsilon(y,s)\Big\}}_{\textbf{A}}\\+&\underbrace{\frac{j}{2\sqrt \epsilon} H_s(1)\sum_{j=2}^K\sum_{A_j\subset I_+} \Big\{v_{|A_j|}^{\epsilon}(A_j,s)\prod_{y\in A_j^c}\rho_\epsilon(y,s)\Big\}}_{\textbf{B}}\\-&\underbrace{\frac{j}{2\sqrt \epsilon} H_s(1)\sum_{x\in I_+\setminus\{N\}}  \Big\{(\bar\eta_{s}(x)-\bar\eta_{s}(N))\prod_{y\in I_+\setminus \{x\}}\rho_\epsilon(y,s)\Big\}}_{\textbf{C}}\\
+&\underbrace{{\frac{j}{2\sqrt \epsilon} H_s(1)\sum_{x\in I_+\setminus\{N\}}  \Big\{(\mathbb E^\epsilon_{\mu^\eps}[\bar\eta_s(x)]-\mathbb E^\epsilon_{\mu^\eps}[\bar\eta_s(N)])\prod_{y\in I_+\setminus \{x\}}\rho_\epsilon(y,s)\Big\}}}_{\textbf{D}}\\
-&\underbrace{\frac{j}{2\sqrt \epsilon} H_s(1)\bar\eta_{s}(N)R_\epsilon(s)}_{\textbf{E}}\\-&\underbrace{\frac{j}{2\sqrt \epsilon} H_s(1)K\bar\eta_{s}(N)(\rho_\epsilon(N,s))^{K-1}}_{\textbf{F}}.
\end{split}
\end{equation*}
We can obtain analogous identities with respect to the left boundary but we omit them. 
Observe that the square of the  $L^2$-norm of the time integral of  {\textbf{A}} can be bounded from above by a constant 
$C_{j,H,K}\epsilon ^\zeta(t^{1+\zeta}\mathbf 1_{t<1}+t^2 \mathbf 1_{t\geq 1})$ for some $\zeta>0$, using Theorem \ref{thm:vestimate3}. Indeed,  since $|\rho_\epsilon(s,y)|\leq 1$ for all $s$ and $y$, we have that
\begin{equation*}
\begin{split}
&\mathbb E^\epsilon_{\mu^\eps} \Big[\Big(\int_0^t\frac{j}{2\sqrt \epsilon} H_s(1)\sum_{j=2}^K\sum_{A_j\subset I_+} \Big\{\prod_{x\in A_j} \bar\eta_{s}(x)\prod_{y\in A_j^c}\rho_\epsilon(y,s)\Big\}\Big)^2\Big]\\
\leq &\sum_{j=2}^K\sum_{A_j\subset I_+} \frac{C_{j,H}}{\epsilon} \int_0^t\int_0^r \mathbb E^\epsilon_{\mu^\eps} \Big[\prod_{x\in A_j}\bar \eta_{s}(x)\prod_{\tilde x\in A_j}\bar \eta_{r}(\tilde x)\Big] dsdr\leq C_{j,H,K}\epsilon^\zeta(t^{1+\zeta}\mathbf 1_{t<1}+t^2 \mathbf 1_{t\geq 1}).
\end{split}
\end{equation*}
Analogously, from Theorem \ref{thm:vestimate3} the $L^2$-norm of the time integral of the term {\bf{B}}  can be bounded from above by $\epsilon$ to some positive exponent. From Corollary \ref{cor:vestimate3_difference}, the $\mathbb L^2$-norm of the time integral of the term {\bf{C}} can be bounded from above by $C_{j,H,K}\epsilon^{1-4\zeta}(t^{1+\zeta}\mathbf 1_{t<1}+t^2 \mathbf 1_{t\geq 1})$ for some $\zeta>0$.  {Since $\mathbb E^\epsilon_{\mu^\eps}[\bar\eta_s(x)]=v^{\eps}_1(x,s)$ and $\mathbb E^\epsilon_{\mu^\eps}[\bar\eta_s(N)]=v^{\eps}_1(N,s)$, by Theorem \ref{thm:vestimate2}, the time integral of the term {\bf{D}} is bounded as follows
$$\int_{0}^t\frac{j}{2\sqrt \epsilon} H_s(1)\sum_{x\in I_+\setminus\{N\}}  \Big\{(\mathbb E^\epsilon_{\mu^\eps}[\bar\eta_s(x)]-\mathbb E^\epsilon_{\mu^\eps}[\bar\eta_s(N)])\prod_{y\in I_+\setminus \{x\}}\rho_\epsilon(y,s)\Big\}ds\leq C_{j,H,K}\frac{\eps^{1-2\zeta}}{\sqrt{\eps}}t,$$
where we use $\prod_{y\in I_+\setminus \{x\}}|\rho_\epsilon(s,y)|\leq 1$.} Finally, 
to compute the $ L^2$-norm of the time integral of the term {\bf{E}} we use  Corollary  C.2 of \cite{birmpa2025nonequilibriumfluctuationsstirringprocess} which says that:

\begin{corollary}\label{corollary: difference of rhos}
Let $x\in \{-N,\dots, N\}$ and $0\leq t\leq T$. There exists a constant $C_T$ which depends on $T$ such that for $0<\eps<1$
\begin{eqnarray*}
\sup_{x,y\in\Lambda_N:|x-y|\leq 1}|\rho_\epsilon(x,t)-\rho_\epsilon(y,t)|\leq C_T \epsilon^{1-2\zeta}
\end{eqnarray*}
\end{corollary}
We leave the details to the reader but it is enough to follow the same steps as presented above for the case $K=2$.  We now return to \eqref{intpartofmart_2}. Combining the
rightmost term in its second line with the third line,
we obtain
\begin{equation}\label{intpartofmart_3}
\begin{split}
-&\frac{1}{2\sqrt{\epsilon}}
\Big(
\partial_u H_s(1)
+jK\rho_s(1)^{K-1}H_s(1)
\Big)\bar\eta_s(N)
+R_s^\epsilon,
\end{split}
\end{equation}
where
\begin{equation*}
R_s^\epsilon
:=
-\frac{jK}{2\sqrt{\epsilon}}H_s(1)
\Big(
\rho_\epsilon(N,s)^{K-1}
-\rho_s(1)^{K-1}
\Big)\bar\eta_s(N).
\end{equation*}

By the Robin boundary condition
\eqref{eq:Robin_equation_line}, the first term in
\eqref{intpartofmart_3} vanishes identically.

Moreover, since $0\leq \rho_\epsilon(N,s),\rho_s(1)\leq 1$,
we have
\[
\left|
\rho_\epsilon(N,s)^{K-1}
-\rho_s(1)^{K-1}
\right|
\leq
(K-1)
\left|
\rho_\epsilon(N,s)-\rho_s(1)
\right|.
\]
Therefore, applying Theorem \ref{thm:vestimate3}
with $n=m=1$, together with the hydrodynamic convergence
\eqref{eq:hydro_conv}, and arguing as in
Lemma \ref{lem:use_of_hydrodynamic_limit}, we obtain
\[
\lim_{\epsilon\to0}
\mathbb E^\epsilon_{\mu^\epsilon}
\left[
\left(
\int_0^t R_s^\epsilon\,ds
\right)^2
\right]=0.
\]

The analogous contribution from the left boundary
is treated in the same way. This completes the
closure of the martingale for arbitrary fixed $K$.

\bibliographystyle{plain}  
\bibliography{references}

\end{document}